\documentclass[12pt]{article} 
\usepackage{amssymb}
\usepackage{amsmath,bm}
\usepackage{array}
\usepackage{graphics}
\usepackage{color}
\usepackage{xcolor}
\usepackage{xypic}
\usepackage{tikz}
\usepackage[utf8]{inputenc}                    
\usepackage{microtype}                         
\usepackage{mathtools, amsthm, amssymb, eucal} 
\usepackage[normalem]{ulem}                    
\usepackage{mdframed}
\usepackage{verbatim}
\usepackage{booktabs} 
\usepackage{mathrsfs}
\usepackage{titling}
\usepackage{xypic}
\usepackage[paper=letterpaper, margin=1in, headsep=20pt]{geometry}
\usepackage{hyperref}
\usepackage{enumerate}
\usepackage{multicol}
    \theoremstyle{plain}
    \newtheorem{thm}{Theorem}[section]
    \newtheorem{lem}[thm]{Lemma}
    \newtheorem{prop}[thm]{Proposition}
    \newtheorem{cor}[thm]{Corollary}

    \theoremstyle{definition}
    \newtheorem{defn}{Definition}[section]
    
    \newtheorem{exmp}{Example}[section]

    \theoremstyle{remark}
    \newtheorem{rem}{Remark}[section]

\input prepictex   \input pictex    \input postpictex

\def\mput #1 #2/{\put{$\bullet_{#1}$} [tl] <-1mm,1mm> at #2}

\begin{document}

\centerline{\textbf{\Large Combinatorics of Exceptional Sequences of type $\mathbb{\tilde{A}}_n$}}
\medskip
\centerline{Ray Maresca} 

\begin{abstract}
\noindent
It is known that there are infinitely many exceptional sequences of quiver representations for Euclidean quivers. In this paper we study those of type $\tilde{\mathbb{A}}_n$ and classify them into finitely many parametrized families. We first give a bijection between exceptional collections and a combinatorial object known as strand diagrams. We will then realize these strand diagrams as chord diagrams and then arc diagrams on an annulus. Using arc diagrams, we will define parametrized families of exceptional collections and use arc diagrams to show that there are finitely many such families. We moreover show that these families of exceptional collections are in bijection with equivalence classes of small arc diagrams. Finally, we provide an algebraic explanation of parametrized families using the transjective component of the bounded derived category.
\end{abstract}

\section{Introduction}

\indent

Exceptional sequences are sequences of modules (equivalently representations of a quiver) that satisfy certain homological properties. They first arose in the study of vector bundles over the projective plane, [\ref{ref: vector bundles over projective plane 10}] and [\ref{ref: vector bundles over projective plane}]. Shortly after, they became a popular object in representation theory first studied in [\ref{ref: Bill Exceptional sequences}] and [\ref{ref: braid group acts on excep seq}]. For instance, Crawley-Boevey showed in [\ref{ref: Bill Exceptional sequences}] that the braid group acts transitively on the set of so-called complete exceptional sequences. Of the many different approaches one can take in the study of these sequences, their combinatorics have been of much interest. The number of so called complete exceptional sequences of Dynkin type has been studied in [\ref{ref: counting exceptional sequences}] and [\ref{ref: Exceptional sequences dynkin type}]. In this paper, we will study complete exceptional sequences of type $\tilde{\mathbb{A}}_n$. The universal cover of these quivers are quivers of type $\mathbb{A}$, whose exceptional sequences have been well studied. For instance there is a bijection between these exceptional sequences and parking functions given in [\ref{ref: bijection with parking functions}]; there is also a bijection between complete exceptional sequences and rooted labeled forests [\ref{ref: Igusa Sen rooted forests}]. Typically, exceptional sequences of type $\mathbb{A}$ are studied using chord diagrams as in [\ref{ref: chord diagrams 1}] and [\ref{ref: chord diagrams 2}]. We will begin by using a slightly different combinatorial object, then show that these are equivalent to a generalization of these chord diagrams.\\

In this paper, we will study exceptional sequences of finite length $\Bbbk Q$-modules where $Q$ is a quiver of type $\tilde{\mathbb{A}}_n$ first by using a combinatorial object introduced in [\ref{ref: combinatorics exceptional sequences type A}] known as strand diagrams. Our first result will be a bijection between exceptional collections and strand diagrams on the universal cover of $Q$, which is an infinitely long type $\mathbb{A}$ quiver. After, we will introduce two more combinatorial objects, chord and arc diagrams. These chord diagrams will be slightly different than those in [\ref{ref: chord diagrams 1}] and [\ref{ref: chord diagrams 2}]. Also, the arc diagrams we present here are very similar to those used in the study of clusters [\ref{ref: clusters are in bijection with triangulations}] and the study of tiling algebras [\ref{ref: BCS}]; however here, we will use a different convention of associating arcs to modules than the one used in the aforementioned papers in order to use them to study a different algebraic object, namely exceptional collections. The comparison of the convention used in [\ref{ref: clusters are in bijection with triangulations}] to study clusters and the one used in this paper is given in [\ref{ref: Igusa Maresca}]. \\

It is well known that there are infinitely many complete exceptional sequences of type $\tilde{\mathbb{A}}_n$; however, we will classify exceptional collections into parametrized families using arc diagrams and define so-called `small' arc diagrams. We will use these small diagrams along with the aforementioned bijection to show that there are only finitely many such families and they are in bijection with equivalence classes of these small diagrams. In [\ref{ref: Igusa Maresca}] we have shown that the number of such diagrams for a particular orientation of the quiver is given by a generalization of the Catalan numbers that count lattice paths in $\mathbb{R}^2$ as defined in [\ref{ref: Gould Generalized Catalan numbers}] and [\ref{ref: Master's Thesis Generalized Catalan Numbers}]. After we define families geometrically, we will provide an algebraic definition for parametrized families in terms of the Auslander--Reiten quiver of the transjective component of the bounded derived category. In Section \ref{sec: examples}, we compute all parametrized families of exceptional sequences of type $\tilde{\mathbb{A}}_2$ and list them in a table along side their corresponding small strand and arc diagrams. \\

\noindent
\textbf{Acknowledgments:} The author would like to thank Kiyoshi Igusa for advising him throughout this project. He would also like to thank Shujian Chen for the helpful discussions along with Simon Huynh for the assistance with constructing the diagrams. He also thanks Eric Hanson for the helpful feedback leading to enhanced readability along with an anonymous referee who helped him both shorten and provide intuition for the proof of Lemma \ref{lem: key lemma on annulus}.

\section{Preliminaries}

\indent

We wish to use the results of [\ref{ref: combinatorics exceptional sequences type A}] to describe exceptional sequences over path algebras of quivers of type $\mathbb{\tilde{A}}_n$. Thus in this section, we will introduce representations of quivers, exceptional modules and sequences, quivers of type $\mathbb{\tilde{A}}_n$, and universal covers of quivers. 

\subsection{Representations of Quivers}

\indent

Throughout, let $\Bbbk$ be an algebraically closed field. A \textbf{quiver} $Q$ is a directed graph. More formally, it is a $4$-tuple $Q = (Q_0,Q_1,s,t)$ where $Q_0$ is the \textbf{set of vertices}, $Q_1$ is the \textbf{set of arrows}, and $s,t:Q_0\rightarrow Q_1$ are maps that assign to each vertex a \textbf{starting} and \textbf{terminal} point respectively. Throughout, we will typically denote the arrows in $Q_1$ by lower case Greek letters and the vertices in $Q_0$ by numbers. A \textbf{path of length} $m \geq1$ in $Q$ is a finite sequence of arrows $\alpha_1\alpha_2\dots\alpha_m$ where $t(\alpha_i) = s(\alpha_j)$ for all $1 \leq j \leq m-1$. We denote the path algebra of $Q$ by $\Bbbk Q$. A \textbf{monomial relation} in $Q$ is given by a path of the form $\alpha_1\alpha_2\dots\alpha_m$ where $m\geq 2$. A two sided ideal of $\Bbbk Q$ is called \textbf{admissible} if $R^m_Q \subset I \subset R^2_Q$ for some $m \geq 2$ where $R_Q$ is the arrow ideal in $\Bbbk Q$. 

A \textbf{$\Bbbk$-representation} $V$ of a quiver $Q$ is an assignment of a $\Bbbk$-vector space $V_i$ for each $i\in Q_0$ and a vector space morphism $\phi_{\alpha}:V_i \rightarrow V_j$ for each $\alpha \in Q_1$ such that $s(\alpha) = i$ and $t(\alpha) = j$. We call a representation $V$ \textbf{finite dimensional} if $V_i$ is a finite dimensional vector space for all $i\in Q_0$. Let $V = (V_i, \phi_{\alpha})$ and $W = (W_i, \psi_{\alpha})$ be two representations of a quiver $Q$. A \textbf{morphism} $\theta: V \rightarrow W$ is a collection of linear maps $\theta_i:V_i\rightarrow W_i$ such that for each $\alpha \in Q_1$, we have $\psi_{\alpha} \circ \theta_{s(\alpha)} = \theta_{t(\alpha)} \circ \phi_{\alpha}$. 

We denote the \textbf{direct sum} of two representations by $V \oplus W = (V_i \oplus W_i, \phi_{\alpha} \oplus \psi_{\alpha})$. We call a representation \textbf{indecomposable} if it is not isomorphic to the direct sum of two nonzero representations. The category rep$_{\Bbbk} (Q) = \text{rep}(Q)$ that consists of finite-dimensional $\Bbbk$-representations of the quiver $Q$ as objects and morphisms of representations as morphisms forms an abelian category. Moreover, the indecomposable representations form a full subcategory denoted ind(rep$_\Bbbk(Q))$. One can equivalently view representations of a quiver $Q$ as modules over its path algebra $\Bbbk Q$, the algebra with basis the set of paths in $Q$ and with multiplication defined by concatenation or zero. It is well known that the two categories are equivalent and as a result, the terms modules and representations will be used interchangeably throughout. As a consequence, we have for any representations $V$ and $W$, well defined notions of Ext$_{\text{rep}(Q)}^i (V,W) = \text{Ext}^i(V,W)$ for all $i \geq 1$ and Hom$_{\text{rep}(Q)}(V,W) =$ Hom$(V,W)$. For more on representation theory of quivers see [\ref{ref: Blue Book}] and [\ref{ref: Schiffler Quiver Reps}].

We call a representation of a quiver $Q$ \textbf{exceptional} if End$_{\text{rep}(Q)}(V)$ is a division ring and Ext$^i (V,V) = 0$ for all $i\geq 1$. When $\Bbbk Q$ is a hereditary algebra, which is the case for quivers of type $\mathbb{\tilde{A}}_n$, the first condition implies that the representation $V$ is indecomposable and the second reduces to Ext$^1 (V,V) =$ Ext$(V,V) = 0$ as higher extension groups vanish. It is important to note that not all indecomposable representations of a quiver of type $\mathbb{\tilde{A}}_n$ are exceptional. An \textbf{exceptional sequence} $\xi = (V_1, \dots V_k)$ is a sequence of exceptional representations of $Q$ such that Hom$(V_i,V_j) = 0 =$ Ext$(V_i,V_j)$ for all $j < i$. An \textbf{exceptional collection} is a \textit{set} of representations $\xi=\{V_1, \dots, V_k\}$ such that the $V_j$ can be ordered in such a way that they make an exceptional sequence. It is well known that for a quiver with $n$ vertices, if $(V_1, \dots V_k)$ is an exceptional sequence, then $k \leq n$ [\ref{ref: Bill Exceptional sequences}]. When $k = n$, we call $\xi$ a \textbf{complete} exceptional sequence (collection). From now on, we adopt the convention that `exceptional sequence (collection)' means complete. For more on exceptional sequences see [\ref{ref: Bill Exceptional sequences}] and [\ref{ref: Exceptional sequences dynkin type}].

\subsection{Quivers of Type $\mathbb{\tilde{A}}_n$}

\indent

A \textbf{quiver of type $\mathbb{\tilde{A}}_n$} is one whose underlying graph is of the form shown in Figure \ref{fig: underlying graph and example of orientation}. To make a graph of type $\tilde{\mathbb{A}}_n$ into a quiver, we define an \textbf{orientation vector} $\bm{\varepsilon} = (\varepsilon_0, \dots , \varepsilon_n) \in \{-,+\}^{n+1}$. We then define $\alpha_i \in Q_1$ as

\vspace{-.7cm}

\begin{center}
\begin{multicols}{2}

 \begin{displaymath}
   \alpha_i = \left\{
     \begin{array}{lr}
       i \rightarrow i+1 & : \varepsilon_i = +\\
       i \leftarrow i+1 & :  \varepsilon_i = -
     \end{array}
   \right.
\end{displaymath}

\columnbreak

 \begin{displaymath}
   \alpha_0 = \left\{
     \begin{array}{lr}
       n+1 \rightarrow 1 & : \varepsilon_0 = +\\
       n+1 \leftarrow 1 & :  \varepsilon_0 = -
     \end{array}
   \right.
\end{displaymath}

\end{multicols}
\end{center}

Note that when considering our vertex set modulo $|Q_0| = n+1$, the convention for $\alpha_0$ is consistent with that of $\alpha_i$. So long as $\varepsilon_i \neq \varepsilon_j$ for some $i$ and $j$, these quivers are hereditary and tame. By \textbf{hereditary}, we mean that submodules of projective $\Bbbk Q$-modules are projective and by \textbf{tame}, we mean there are infinitely many indecomposable $\Bbbk Q$-modules and for all $n\in\mathbb{N}$, all but finitely many isomorphism classes of $n$-dimensional indecomposables occur in a finite number of one-parameter families. It is known that the module category, hence the Auslander--Reiten quiver $\Gamma_{\Bbbk Q}$, of a tame hereditary algebra can be partitioned into three sections, the preprojective, regular and preinjective components. For $\tau$ the Auslander--Reiten translate, the preprojective and preinjective components are defined as follows. 

\begin{figure}
    \centering
    \tikzset{every picture/.style={line width=0.75pt}} %set default line width to 0.75pt        

\begin{tikzpicture}[x=0.75pt,y=0.75pt,yscale=-1,xscale=1]
%uncomment if require: \path (0,332); %set diagram left start at 0, and has height of 332

%Straight Lines [id:da07376109030777367] 
\draw    (93.22,224.01) -- (108.41,258.18) ;
\draw [shift={(109.22,260.01)}, rotate = 246.04] [color={rgb, 255:red, 0; green, 0; blue, 0 }  ][line width=0.75]    (10.93,-3.29) .. controls (6.95,-1.4) and (3.31,-0.3) .. (0,0) .. controls (3.31,0.3) and (6.95,1.4) .. (10.93,3.29)   ;
%Straight Lines [id:da36420736141886767] 
\draw    (163.22,270.01) -- (122.22,270.01) ;
\draw [shift={(120.22,270.01)}, rotate = 360] [color={rgb, 255:red, 0; green, 0; blue, 0 }  ][line width=0.75]    (10.93,-3.29) .. controls (6.95,-1.4) and (3.31,-0.3) .. (0,0) .. controls (3.31,0.3) and (6.95,1.4) .. (10.93,3.29)   ;
%Straight Lines [id:da1930773843440663] 
\draw    (173.22,262.01) -- (193.27,224.77) ;
\draw [shift={(194.22,223.01)}, rotate = 118.3] [color={rgb, 255:red, 0; green, 0; blue, 0 }  ][line width=0.75]    (10.93,-3.29) .. controls (6.95,-1.4) and (3.31,-0.3) .. (0,0) .. controls (3.31,0.3) and (6.95,1.4) .. (10.93,3.29)   ;
%Straight Lines [id:da09373006233409709] 
\draw    (187.22,210.01) -- (150.75,179.3) ;
\draw [shift={(149.22,178.01)}, rotate = 40.1] [color={rgb, 255:red, 0; green, 0; blue, 0 }  ][line width=0.75]    (10.93,-3.29) .. controls (6.95,-1.4) and (3.31,-0.3) .. (0,0) .. controls (3.31,0.3) and (6.95,1.4) .. (10.93,3.29)   ;
%Straight Lines [id:da22912725092722708] 
\draw    (97.22,205.01) -- (135.54,180.1) ;
\draw [shift={(137.22,179.01)}, rotate = 146.98] [color={rgb, 255:red, 0; green, 0; blue, 0 }  ][line width=0.75]    (10.93,-3.29) .. controls (6.95,-1.4) and (3.31,-0.3) .. (0,0) .. controls (3.31,0.3) and (6.95,1.4) .. (10.93,3.29)   ;
%Straight Lines [id:da7973143167246244] 
\draw    (68.22,108.01) -- (106.22,108.01) ;
%Straight Lines [id:da35232664267191005] 
\draw    (125.22,108.01) -- (147.22,108.01) ;
%Straight Lines [id:da6511432169365561] 
\draw    (207.22,108.01) -- (234.22,108.01) ;
%Straight Lines [id:da6103315884609939] 
\draw    (71.22,99.01) -- (141.22,55.01) ;
%Straight Lines [id:da13499817684123938] 
\draw    (165.22,56.01) -- (232.22,101.01) ;

% Text Node
\draw (86,206.4) node [anchor=north west][inner sep=0.75pt]    {$1$};
% Text Node
\draw (78,240.4) node [anchor=north west][inner sep=0.75pt]    {$\alpha_1$};
% Text Node
\draw (108,261.4) node [anchor=north west][inner sep=0.75pt]    {$2$};
% Text Node
\draw (140,275.4) node [anchor=north west][inner sep=0.75pt]    {$\alpha_2$};
% Text Node
\draw (164,260.4) node [anchor=north west][inner sep=0.75pt]    {$3$};
% Text Node
\draw (188,240.4) node [anchor=north west][inner sep=0.75pt]    {$\alpha_3$};
% Text Node
\draw (188.22,206.41) node [anchor=north west][inner sep=0.75pt]    {$4$};
% Text Node
\draw (165,180.4) node [anchor=north west][inner sep=0.75pt]    {$\alpha_4$};
% Text Node
\draw (138,168.4) node [anchor=north west][inner sep=0.75pt]    {$5$};
% Text Node
\draw (100,180.4) node [anchor=north west][inner sep=0.75pt]    {$\alpha_5$};
% Text Node
\draw (58,99.4) node [anchor=north west][inner sep=0.75pt]    {$1$};
% Text Node
\draw (111,99.4) node [anchor=north west][inner sep=0.75pt]    {$2$};
% Text Node
\draw (234,98.4) node [anchor=north west][inner sep=0.75pt]    {$n$};
% Text Node
\draw (133,35.4) node [anchor=north west][inner sep=0.75pt]    {$n+1$};
% Text Node
\draw (163,102.4) node [anchor=north west][inner sep=0.75pt]    {$\dotsc $};
% Text Node
\draw (309,70) node [anchor=north west][inner sep=0.75pt]   [align=left] {The underlying graph of a quiver of type $\tilde{\mathbb{A}}_n$};
% Text Node
\draw (316,218) node [anchor=north west][inner sep=0.75pt]   [align=left] {A quiver of type $\tilde{\mathbb{A}}_4$ with $\bm{\varepsilon}= (-,+,-,+,+)$};
\end{tikzpicture}
\caption{An Example of $Q^{\bm\varepsilon}$}
\label{fig: underlying graph and example of orientation}
\end{figure}
\begin{defn}{\color{white} .}
\begin{itemize}
\item A connected component $C$ of the Auslander-Reiten quiver $\Gamma_{\Bbbk Q}$ of $\Bbbk Q$ is a \textbf{preprojective component} if the following hold:
\begin{enumerate}
\item Each indecomposable module $M\in C$ is isomorphic to $\tau_{\Bbbk Q}^{-i} (P)$ for some $i\geq 0$ and some indecomposable projective $\Bbbk Q$-module $P$. 
\item $C$ does not have any oriented cycles. 
\end{enumerate}

\item A connected component $C$ of the Auslander-Reiten quiver $\Gamma_{\Bbbk Q}$ of $\Bbbk Q$ is a \textbf{preinjective component} if the following hold:
\begin{enumerate}
\item Each indecomposable module $M\in C$ is isomorphic to $\tau_{\Bbbk Q}^{i} (I)$ for some $i\geq 0$ and some indecomposable injective $\Bbbk Q$-module $I$. 
\item $C$ does not have any oriented cycles. 
\end{enumerate}
\end{itemize}
\end{defn}

For $\mathbb{\tilde{A}}_n$ quivers, the regular component consists of the left, right and homogeneous tubes, all of which are stable under AR translation $\tau$. A regular module with no proper regular submodules is called \textbf{quasi-simple} or \textbf{simple regular}. The number of quasi-simple modules in a tube is called the \textbf{rank} of the tube. For any regular module $M$, there exists a chain $$ 0=M_0 \subset M_1 \subset \dots\subset M_l = M $$ of regular submodules of $M$ with $M_i/M_{i-1}$ a quasi-simple module for any $1\leq i \leq l$, which we denote by $rl(M)$. It is well known that in type $\tilde{\mathbb{A}}$, the preprojective and preinjective modules are exceptional. To classify which regular modules are exceptional, the following lemma, whose proof along with a more general statement can be found in [\ref{ref: blue book 2}], is useful.

\begin{lem}\label{lem: hom and ext for regulars}
Let $A = \Bbbk Q$ for a Euclidean quiver $Q$, $T$ be a stable tube of rank $r\geq1$, and $M$ be an indecomposable module in $T$.
\begin{enumerate}
\item If $k\geq0$ is an integer such that $kr<rl(M)\leq(k+1)r$, then $\text{dim}_\Bbbk\text{End}(M) = k+1$.
\item If $k\geq0$ is an integer such that $kr<rl(M)\leq(k+1)r$, then $\text{dim}_\Bbbk\text{Ext}(M,M) = k$.
\end{enumerate}
\end{lem}

Moreover, the path algebras are also string algebras, simplifying the aforementioned tripartite classification of $\Gamma_{\Bbbk Q}$. For more on representation theory of euclidean quivers see [\ref{ref: Blue Book}] and [\ref{ref: blue book 2}].

\subsection{String Algebras/Gentle Algebras}

Throughout the paper, we will appeal to the fact that the path algebra of a quiver of type $\tilde{\mathbb{A}}$ is a gentle algebra.

\begin{defn}
For an admissible ideal $I$, the algebra $B = \Bbbk Q/ I$ is a \textbf{string algebra} if
\begin{enumerate}
\item At each vertex of $Q$, there are at most two incoming arrows and at most two outgoing arrows.
\item For each arrow $\beta$ there is at most one arrow $\alpha$ and at most one arrow $\gamma$ such that $\alpha\beta \notin I$ and $\beta\gamma \notin I$.  \\

If moreover, we have the following two conditions, the string algebra $B$ is called \textbf{gentle}.
\item The ideal $I$ is generated by a set of monomials of length two.
\item For every arrow $\alpha$, there is at most one $\beta$ and one $\gamma$ such that $0 \neq \alpha\beta \in I$ and $0 \neq \gamma\alpha \notin I$.
\end{enumerate}
\end{defn}

It is well known that the indecomposable modules over string algebras are either string or band modules [\ref{ref: String Algebra Info}]. For $\Bbbk Q/I$ a sting algebra, to define string modules, we first define for $a\in Q_1$ a \textbf{formal inverse} $\alpha^{-1}$, such that $s(\alpha^{-1}) = t(\alpha)$ and $t(\alpha^{-1}) = s(\alpha)$. Let $Q_1^{-1}$ denote the set of formal inverses of arrows in $Q_1$. We define a \textbf{walk} as a sequence $\omega = \omega_0\dots \omega_r$ such that for all $i\in\{0,1,\dots,r\}$, we have $t(\omega_i) = s(\omega_{i+1})$ where $\omega_i \in Q_1 \cup Q_1^{-1}$. A \textbf{string} is a walk $\omega$ with no sub-walk $\alpha\alpha^{-1}$ or $\alpha^{-1}\alpha$. A \textbf{band} $\beta = \beta_1\dots \beta_n$ is a cyclic string, that is, $t(\beta_n) = s(\beta_1)$. We define the \textbf{start or beginning} of a string $S = \omega_0\dots \omega_r$, denoted by $s(S)$, as $s(\omega_0)$. Similarly, we define the \textbf{end} of a string $S$, denoted by $t(S)$, as $t(\omega_r)$. For quivers of type $\mathbb{\tilde{A}}_n$ the band modules lie in the homogeneous tubes and we can classify in which component of the Auslander-Reiten quiver the string modules reside by their shape, as we will see in Remark \ref{rem: prepro and preinj strands defn}. For quivers of type $\tilde{\mathbb{A}}$, we take the convention that all named strings move in the counter-clockwise direction around the quiver $Q$. We denote by $ij_k$ the string module of length $k$ associated to the walk $\omega_1\cdots \omega_{k-1}$ where $s(\omega_1)=i+1$ and $t(\omega_{k-1})=j$. For $k=1$, $(j-1)j_1$ denotes the simple module at vertex $j$, associated to the walk $e_j$ where $e_j$ is the lazy path at vertex $j$. Note that beginning at the previous vertex is not the usual convention, however it will be much more convenient for this paper since it helps in classifying extensions between string modules. When drawing the graph associated to strings, we take the convention that the head of each arrow is at the bottom. An example of a graph of a string and its associated representation can be seen in Figure \ref{fig: example of string module}.

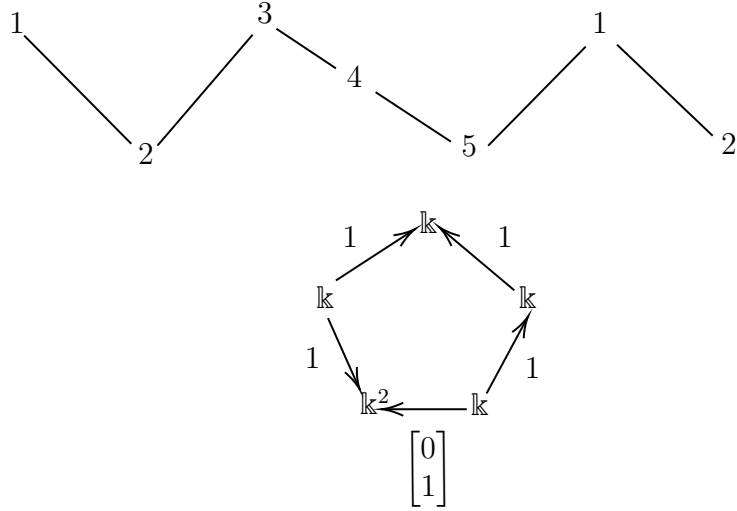
\begin{figure}[h!]
    \centering

\tikzset{every picture/.style={line width=0.75pt}} %set default line width to 0.75pt        

\begin{tikzpicture}[x=0.75pt,y=0.75pt,yscale=-1,xscale=1]
%uncomment if require: \path (0,332); %set diagram left start at 0, and has height of 332

%Straight Lines [id:da07376109030777367] 
\draw    (263.22,176.01) -- (278.41,210.18) ;
\draw [shift={(279.22,212.01)}, rotate = 246.04] [color={rgb, 255:red, 0; green, 0; blue, 0 }  ][line width=0.75]    (10.93,-3.29) .. controls (6.95,-1.4) and (3.31,-0.3) .. (0,0) .. controls (3.31,0.3) and (6.95,1.4) .. (10.93,3.29)   ;
%Straight Lines [id:da36420736141886767] 
\draw    (333.22,222.01) -- (292.22,222.01) ;
\draw [shift={(290.22,222.01)}, rotate = 360] [color={rgb, 255:red, 0; green, 0; blue, 0 }  ][line width=0.75]    (10.93,-3.29) .. controls (6.95,-1.4) and (3.31,-0.3) .. (0,0) .. controls (3.31,0.3) and (6.95,1.4) .. (10.93,3.29)   ;
%Straight Lines [id:da1930773843440663] 
\draw    (343.22,214.01) -- (363.27,176.77) ;
\draw [shift={(364.22,175.01)}, rotate = 118.3] [color={rgb, 255:red, 0; green, 0; blue, 0 }  ][line width=0.75]    (10.93,-3.29) .. controls (6.95,-1.4) and (3.31,-0.3) .. (0,0) .. controls (3.31,0.3) and (6.95,1.4) .. (10.93,3.29)   ;
%Straight Lines [id:da09373006233409709] 
\draw    (357.22,162.01) -- (320.75,131.3) ;
\draw [shift={(319.22,130.01)}, rotate = 40.1] [color={rgb, 255:red, 0; green, 0; blue, 0 }  ][line width=0.75]    (10.93,-3.29) .. controls (6.95,-1.4) and (3.31,-0.3) .. (0,0) .. controls (3.31,0.3) and (6.95,1.4) .. (10.93,3.29)   ;
%Straight Lines [id:da22912725092722708] 
\draw    (267.22,157.01) -- (305.54,132.1) ;
\draw [shift={(307.22,131.01)}, rotate = 146.98] [color={rgb, 255:red, 0; green, 0; blue, 0 }  ][line width=0.75]    (10.93,-3.29) .. controls (6.95,-1.4) and (3.31,-0.3) .. (0,0) .. controls (3.31,0.3) and (6.95,1.4) .. (10.93,3.29)   ;
%Straight Lines [id:da6089440174508203] 
\draw    (177.22,88.96) -- (225.22,32.96) ;
%Straight Lines [id:da4831202501374725] 
\draw    (237.22,29.96) -- (267.22,49.96) ;
%Straight Lines [id:da16050645757980164] 
\draw    (287.22,61.96) -- (325.22,86.96) ;
%Straight Lines [id:da4846225605325867] 
\draw    (344,89) -- (393.22,38.96) ;
%Straight Lines [id:da49472559013866757] 
\draw    (409,38) -- (457.22,83.96) ;
%Straight Lines [id:da4173917380858625] 
\draw    (110,33.4)-- (164,88) ;

% Text Node
\draw (256,158.4) node [anchor=north west][inner sep=0.75pt]    {$\Bbbk$};
% Text Node
\draw (278,210.4) node [anchor=north west][inner sep=0.75pt]    {$\Bbbk^2$};
% Text Node
\draw (334,212.4) node [anchor=north west][inner sep=0.75pt]    {$\Bbbk$};
% Text Node
\draw (358.22,158.41) node [anchor=north west][inner sep=0.75pt]    {$\Bbbk$};
% Text Node
\draw (308,120.4) node [anchor=north west][inner sep=0.75pt]    {$\Bbbk$};
% Text Node
\draw (269,128.4) node [anchor=north west][inner sep=0.75pt]    {$1$};
% Text Node
\draw (250,189.4) node [anchor=north west][inner sep=0.75pt]    {$1$};
% Text Node
\draw (347,128.4) node [anchor=north west][inner sep=0.75pt]    {$1$};
% Text Node
\draw (361,194.4) node [anchor=north west][inner sep=0.75pt]    {$1$};
% Text Node
\draw (299.81,231.4) node [anchor=north west][inner sep=0.75pt]  [xslant=-0.01]  {$\begin{bmatrix}
0\\
1
\end{bmatrix}$};
% Text Node
\draw (101,20.4) node [anchor=north west][inner sep=0.75pt]    {$1$};
% Text Node
\draw (166,86.4) node [anchor=north west][inner sep=0.75pt]    {$2$};
% Text Node
\draw (226,15.4) node [anchor=north west][inner sep=0.75pt]    {$3$};
% Text Node
\draw (271.22,47.36) node [anchor=north west][inner sep=0.75pt]    {$4$};
% Text Node
\draw (329,82.4) node [anchor=north west][inner sep=0.75pt]    {$5$};
% Text Node
\draw (395,20.4) node [anchor=north west][inner sep=0.75pt]    {$1$};
% Text Node
\draw (460,81.4) node [anchor=north west][inner sep=0.75pt]    {$2$};

\end{tikzpicture}

\caption{The graph of a $\displaystyle 12_{6}$ preprojective string associated to the walk $\alpha_2^{-1}\alpha_3\alpha_4\alpha_5^{-1}\alpha_1$ over the path algebra of the quiver in Figure \ref{fig: underlying graph and example of orientation} and its corresponding indecomposable representation.}
    \label{fig: example of string module}
\end{figure}

String modules are useful, since their combinatorial structure can be used to describe both the AR translate $\tau$ [\ref{ref: Butler, Ringel}], and morphisms/extensions between string modules. To provide the description of the AR translate, we need some definitions. The following lemma follows from the definition of a string algebra.

\begin{lem}
Let $S$ be a string of positive length and let $\varepsilon \in \{ Q, Q^{-1} \}$. There is at most one way to add an arrow preceding $s(S)$ whose orientation agrees with $\varepsilon$, such that the resulting walk is still a string. Similarly there is at most one way to add such an arrow following $t(S)$.
\end{lem}

If a string $S$ can be extended at its end by a direct arrow, then we can add a direct arrow at $t(S)$, followed by adding as many inverse arrows as possible. This operation is called \textbf{adding a cohook} at $t(S)$. Similarly, if $S$ can be extended by an inverse arrow at $s(S)$, then we can add an inverse arrow at $s(S)$, followed by as many direct arrows as possible. This operation is called \textbf{adding a cohook} at $s(S)$. The inverse operation is called \textbf{deleting a cohook}. To do this, we find the last direct arrow in $S$ and remove it along with all the subsequent inverse arrows. This operation is called \textbf{deleting a cohook} at the end of $S$. Similarly, we can find the first inverse arrow in $S$ and remove it along with all the preceding direct arrows. This operation is called \textbf{deleting a cohook} at the start of $S$. Note that deleting a cohook at the end (start) of $S$ may not be defined if $S$ does not any direct (inverse) arrows.

There are dual notions to this, namely \textbf{adding a hook} and \textbf{deleting a hook} respectively. If a string $S$ can be extended at its start by a direct arrow, we add the direct arrow at $s(S)$ followed by as many inverse arrows as possible. This operation is called \textbf{adding a hook} at the start of $S$. If a string $S$ can be extended at its end by adding an inverse arrow, we add this inverse arrow to $S$ along with as many direct arrows as possible. This operation is called \textbf{adding a hook} at the end of $S$. To \textbf{delete a hook} from the end of $S$, we find the last inverse arrow in $S$, and remove it along with all its subsequent direct arrows. Analogously, to \textbf{delete a hook} from the start of $S$, we find the first direct arrow of $S$ and remove it along with all preceding inverse arrows. Again, note that deleting a hook at the end (start) of $S$ may not be defined if $S$ does not any inverse (direct) arrows. It is known that in type $\tilde{\mathbb{A}}$, all irreducible morphisms in the preprojective (preinjective) component are given by adding hooks (deleting cohooks). With these combinatorial notions, the following theorem was proven in [\ref{ref: Butler, Ringel}].

\begin{thm}\label{thm: tau of string algebra}
Let $B = \Bbbk Q/I$ be a string algebra, and let $S$ be a string. At either end of $S$, if it is possible to add a cohook, do it. Then, at the ends at which it was not possible to add a cohook, delete a hook. The result is $\tau_BS$. Inversely, at either end of $S$, if it is possible to delete a cohook, do it. Then, at the ends at which it was not possible to delete a cohook, add a hook. The result is $\tau^{-1}_BS$. \hfill \qed
\end{thm}

A basis for the space of all morphisms between two strings was given by Crawley-Boevey in [\ref{ref: CB strings}]. Schr{\"o}er then reformulated the basis given by Crawley-Boevey in [\ref{ref: Schroer}]. This reformulation was then used in [\ref{ref: BDMTY}] to provide a basis for the vectorspace of extensions between two strings. Since we will be using these results of [\ref{ref: BDMTY}] throughout the paper, we will recite some of them here. The following definition follows Schr{\"o}er's reformulation in [\ref{ref: Schroer}].

\begin{defn}
For $C$ a string in the quiver $Q$, the set of all factorizations of $C$ is $$\mathcal{P}(C) = \{(F,E,D) : F, E, \text{ and } D \text{ are strings of $Q$ and $C = FED$} \}.$$ A triple $(F,E,D) \in \mathcal{P}(C)$ is called a \textbf{quotient factorization} of $C$ if the following hold:
\begin{enumerate}
\item $D = e_{t(E)}$ or $D = \gamma D'$ with $\gamma \in Q_1$.
\item $F = e_{s(E)}$ or $F = F'\theta$ with $\theta \in Q_1^{-1}$.
\end{enumerate} 
We denote the set of all quotient factorizations of $C$ by $\mathcal{F}(C)$. A triple $(F,E,D) \in \mathcal{P}(C)$ is called a \textbf{submodule factorization} of $C$ if the following hold:
\begin{enumerate}
\item $D = e_{t(E)}$ or $D = \gamma D'$ with $\gamma \in Q_1^{-1}$.
\item $F = e_{s(E)}$ or $F = F'\theta$ with $\theta \in Q_1$.
\end{enumerate} 
We denote the set of all submodule factorizations of $C$ by $\mathcal{S}(C)$.
\end{defn}

A quotient factorization induces a module morphism from $C$ onto $E$ and dually, a submodule factorization induces a monomorphism from $E$ into $C$. 

\begin{defn}
For $C_1$ and $C_2$ strings in $Q$, a pair $((F_1,E_1,D_1),(F_2,E_2,D_2)) \in \mathcal{F}(C_1) \times \mathcal{S}(C_2)$ is called \textbf{admissible} if $E_1 = E_2$ or $E_1 = E_2^{-1}$.
\end{defn}

Each admissible pair $T = ((F_1,E_1,D_1),(F_2,E_2,D_2))$ provides a morphism $f_T:C_1 \rightarrow C_2$ achieved by projecting onto $E_1$ followed by identifying $E_1$ with $E_2$, then including $E_2$ into $C_2$. These morphisms will be called \textbf{graph maps}. Actually, more is true. These maps form a basis for the hom space as shown by Crawley-Boevey in $[\ref{ref: CB strings}]$.

\begin{thm}\label{thm: basis for hom}
If $A = \Bbbk Q /I$ is a string algebra and $C_1$ and $C_2$ are string modules, the set of graph maps is a basis for Hom$_A(C_1,C_2)$.
\end{thm} 

Using this result, the extension group between two strings in a gentle algebra was classified in [\ref{ref: BDMTY}]. To state this result, we need some more definitions.

\begin{defn}
A graph map $f_T$ given by an admissible pair $T = ((F_1,E_1,D_1),(F_2,E_2,D_2))$ is called \textbf{two-sided} if at least one of $D_1$ and $D_2$ has positive length, and the same holds for $F_1$ and $F_2$. We say that the string $C_1$ is \textbf{connectable} to the string $C_2$ if there exists an arrow $\alpha$ such that $C_1\alpha C_2$ is a string in $Q$, or dually that $C_2\alpha^{-1}C_1$ is a string in $Q$.
\end{defn}

Notice, if $C_1$ is connectable to $C_2$, we get a non-zero extension $$0 \rightarrow C_2 \rightarrow C_1\alpha C_2 \rightarrow C_1 \rightarrow 0.$$ Moreover, as shown in [\ref{ref: Schroer}], two-sided graph maps $f_T\in\text{Hom}_A(C_2,C_1)$ coming from an admissible pair $T = ((F_2,E_2,D_2),(F_1,E_1,D_1))$ give rise to a non-zero extension $$0\rightarrow C_2 \rightarrow F_2ED_1 \oplus F_1ED_2 \rightarrow C_1 \rightarrow 0.$$ Not only this, in [\ref{ref: BDMTY}] it was shown that these extensions form a basis.

\begin{thm}\label{thm: basis for ext}
Let $\Bbbk Q/I$ be a gentle algebra. For strings $C_1$ and $C_2$ in $Q$, the set of extensions coming both from connections $\alpha$ of $C_1$ to $C_2$ and two-sided graph maps $F_T \in \text{Hom}(C_2,C_1)$ form a basis for Ext$^1_{A}(C_1,C_2)$.
\end{thm}

Since path algebras of quivers of type $\tilde{\mathbb{A}}$ are gentle, Theorems \ref{thm: tau of string algebra}, \ref{thm: basis for hom}, and \ref{thm: basis for ext} will be used throughout the paper.

\subsection{Universal Covers of Quivers}

\indent

We begin by recalling some definitions from category theory. We define a \textbf{$\Bbbk$-Category} $\mathcal{C}$ as a category in which the morphism sets are $\Bbbk$-vector spaces and the composition maps are $\Bbbk$-bilinear. A $\Bbbk$-linear functor $F:\mathcal{C} \rightarrow \mathcal{D}$ between two $\Bbbk$-categories is called a \textbf{covering functor} if the following induced maps are bijections for all $M\in\text{Ob}(\mathcal{C})$ and $B\in\text{Ob}(\mathcal{D})$: 

\begin{center}
\begin{tabular}{c c}

 $\underset{\{N : F(N) = B\}}{\coprod} \text{Mor}_{\mathcal{C}}(M,N) \rightarrow \text{Mor}_{\mathcal{D}}(F(M),B)$

&

$\underset{\{N : F(N) = B\}}{\coprod} \text{Mor}_{\mathcal{C}}(N,M) \rightarrow \text{Mor}_{\mathcal{D}}(B,F(M))$

\end{tabular}
\end{center}

\noindent
For more on universal covers and covering functors see [\ref{ref: Gabriel Universal Cover}]. It is well known that the universal cover of quivers of type $\mathbb{\tilde{A}}_n$ are infinite quivers of type $\mathbb{A}$ that preserve the orientation. That is, if we are given an orientation $\bm{\varepsilon} = (\varepsilon_0,\dots,\varepsilon_n)$ of a quiver of type $\mathbb{\tilde{A}}_n$, then the universal cover $\tilde{Q}$ will be a quiver of type $\mathbb{A}_{\infty}$ with orientation $\tilde{\bm
{\varepsilon}} = (\tilde{\varepsilon_i})_{i \in \mathbb{Z}}$ where $\tilde{\varepsilon_i} = \varepsilon_{i\,(\text{mod} \, n+1)}$. 

\begin{figure}[h!]
    \centering

\tikzset{every picture/.style={line width=0.75pt}} %set default line width to 0.75pt        

\begin{tikzpicture}[x=0.75pt,y=0.75pt,yscale=-1,xscale=1]
%uncomment if require: \path (0,332); %set diagram left start at 0, and has height of 332

%Straight Lines [id:da9453382514017628] 
\draw    (256.22,117.96) -- (287.22,117.96) ;
\draw [shift={(289.22,117.96)}, rotate = 180] [color={rgb, 255:red, 0; green, 0; blue, 0 }  ][line width=0.75]    (10.93,-3.29) .. controls (6.95,-1.4) and (3.31,-0.3) .. (0,0) .. controls (3.31,0.3) and (6.95,1.4) .. (10.93,3.29)   ;
%Straight Lines [id:da059807364663140516] 
\draw    (357.22,117.96) -- (389.22,117.96) ;
\draw [shift={(391.22,117.96)}, rotate = 180] [color={rgb, 255:red, 0; green, 0; blue, 0 }  ][line width=0.75]    (10.93,-3.29) .. controls (6.95,-1.4) and (3.31,-0.3) .. (0,0) .. controls (3.31,0.3) and (6.95,1.4) .. (10.93,3.29)   ;
%Straight Lines [id:da32675369167329626] 
\draw    (417.22,117.96) -- (448.22,117.96) ;
\draw [shift={(450.22,117.96)}, rotate = 180] [color={rgb, 255:red, 0; green, 0; blue, 0 }  ][line width=0.75]    (10.93,-3.29) .. controls (6.95,-1.4) and (3.31,-0.3) .. (0,0) .. controls (3.31,0.3) and (6.95,1.4) .. (10.93,3.29)   ;
%Straight Lines [id:da7251529369871446] 
\draw    (527.22,117.96) -- (556.22,117.96) ;
\draw [shift={(558.22,117.96)}, rotate = 180] [color={rgb, 255:red, 0; green, 0; blue, 0 }  ][line width=0.75]    (10.93,-3.29) .. controls (6.95,-1.4) and (3.31,-0.3) .. (0,0) .. controls (3.31,0.3) and (6.95,1.4) .. (10.93,3.29)   ;
%Straight Lines [id:da09604427947784688] 
\draw    (153.22,117.96) -- (186.22,117.96) ;
\draw [shift={(188.22,117.96)}, rotate = 180] [color={rgb, 255:red, 0; green, 0; blue, 0 }  ][line width=0.75]    (10.93,-3.29) .. controls (6.95,-1.4) and (3.31,-0.3) .. (0,0) .. controls (3.31,0.3) and (6.95,1.4) .. (10.93,3.29)   ;
%Straight Lines [id:da4730439742743491] 
\draw    (95.22,117.96) -- (126.22,117.96) ;
\draw [shift={(128.22,117.96)}, rotate = 180] [color={rgb, 255:red, 0; green, 0; blue, 0 }  ][line width=0.75]    (10.93,-3.29) .. controls (6.95,-1.4) and (3.31,-0.3) .. (0,0) .. controls (3.31,0.3) and (6.95,1.4) .. (10.93,3.29)   ;
%Straight Lines [id:da34957018690625064] 
\draw    (237.22,117.96) -- (211.22,117.96) ;
\draw [shift={(209.22,117.96)}, rotate = 360] [color={rgb, 255:red, 0; green, 0; blue, 0 }  ][line width=0.75]    (10.93,-3.29) .. controls (6.95,-1.4) and (3.31,-0.3) .. (0,0) .. controls (3.31,0.3) and (6.95,1.4) .. (10.93,3.29)   ;
%Straight Lines [id:da30963508460718914] 
\draw    (337.22,117.96) -- (311.22,117.96) ;
\draw [shift={(309.22,117.96)}, rotate = 360] [color={rgb, 255:red, 0; green, 0; blue, 0 }  ][line width=0.75]    (10.93,-3.29) .. controls (6.95,-1.4) and (3.31,-0.3) .. (0,0) .. controls (3.31,0.3) and (6.95,1.4) .. (10.93,3.29)   ;
%Straight Lines [id:da4800195139622725] 
\draw    (506.22,117.96) -- (480.22,117.96) ;
\draw [shift={(478.22,117.96)}, rotate = 360] [color={rgb, 255:red, 0; green, 0; blue, 0 }  ][line width=0.75]    (10.93,-3.29) .. controls (6.95,-1.4) and (3.31,-0.3) .. (0,0) .. controls (3.31,0.3) and (6.95,1.4) .. (10.93,3.29)   ;
%Straight Lines [id:da9088907309688203] 
\draw    (606.22,117.96) -- (580.22,117.96) ;
\draw [shift={(578.22,117.96)}, rotate = 360] [color={rgb, 255:red, 0; green, 0; blue, 0 }  ][line width=0.75]    (10.93,-3.29) .. controls (6.95,-1.4) and (3.31,-0.3) .. (0,0) .. controls (3.31,0.3) and (6.95,1.4) .. (10.93,3.29)   ;
%Straight Lines [id:da05210534420689572] 
\draw    (69.22,117.96) -- (43.22,117.96) ;
\draw [shift={(41.22,117.96)}, rotate = 360] [color={rgb, 255:red, 0; green, 0; blue, 0 }  ][line width=0.75]    (10.93,-3.29) .. controls (6.95,-1.4) and (3.31,-0.3) .. (0,0) .. controls (3.31,0.3) and (6.95,1.4) .. (10.93,3.29)   ;

% Text Node
\draw (241,110.4) node [anchor=north west][inner sep=0.75pt]    {$1$};
% Text Node
\draw (293,110.4) node [anchor=north west][inner sep=0.75pt]    {$2$};
% Text Node
\draw (341,110.4) node [anchor=north west][inner sep=0.75pt]    {$3$};
% Text Node
\draw (401,110.4) node [anchor=north west][inner sep=0.75pt]    {$4$};
% Text Node
\draw (460,110.4) node [anchor=north west][inner sep=0.75pt]    {$5$};
% Text Node
\draw (514,110.4) node [anchor=north west][inner sep=0.75pt]    {$6$};
% Text Node
\draw (565,110.4) node [anchor=north west][inner sep=0.75pt]    {$7$};
% Text Node
\draw (192,110.4) node [anchor=north west][inner sep=0.75pt]    {$0$};
% Text Node
\draw (130,110.4) node [anchor=north west][inner sep=0.75pt]    {$-1$};
% Text Node
\draw (70,110.4) node [anchor=north west][inner sep=0.75pt]    {$-2$};
% Text Node
\draw (6,120) node [anchor=north west][inner sep=0.75pt]    {$\dotsc $};
% Text Node
\draw (612,120) node [anchor=north west][inner sep=0.75pt]    {$\dotsc $};

\end{tikzpicture}

\caption{The universal cover of the quiver in Figure \ref{fig: underlying graph and example of orientation}}
    \label{fig: example of universal cover}
\end{figure}

Until stated otherwise, fix $Q^{\bm{\varepsilon}}$ to be a quiver of type $\mathbb{\tilde{A}}_{n-1}$ with orientation $\bm{\varepsilon}$, so that $|Q_0| = n$. Consider the universal cover $\tilde{Q}^{\tilde{\bm{\varepsilon}}}$, the covering functor $F:\Bbbk \tilde{Q}^{\tilde{\bm{\varepsilon}}} \rightarrow  \Bbbk Q^{\bm{\varepsilon}}$, and a $\Bbbk Q^{\bm{\varepsilon}}$-string module $M = ij_k$. Note that $F^{-1}(M) = \{\Bbbk \tilde{Q}^{\tilde{\bm{\varepsilon}}}-\text{string modules} \, i'j'_k : i' \equiv i (\text{mod}\, n), \, j'\equiv j (\text{mod} \, n), \, i'<j', \, \text{and} \, j' - i' = k\}$ is a countable set of $\Bbbk \tilde{Q}^{\tilde{\bm{\varepsilon}}}$-string modules. We must first establish a convention for translating between strings in $Q^{\bm{\varepsilon}}$ and strings in $\tilde{Q}^{\tilde{\bm{\varepsilon}}}$. We define the \textbf{fundamental domain} of the universal cover $\tilde{Q}^{\tilde{\bm{\varepsilon}}}$ to be the subset of vertices $FD = \{0,1,2,\dots, n\}$. 

\begin{defn} \label{defn: fundamental lift}
Consider the string module $ij_k$. We define the \textbf{fundamental lift} of $ij_k$, denoted by $\textbf{fl}(ij_k)$ as follows. 
\begin{itemize}
\item If $ij_k$ is preprojective or left regular, then $\textbf{fl}(ij_k)$ is the string that begins at $\tilde{i} = i \in FD$ and ends at $\tilde{j} = i + k$.
\item If $ij_k$ is preinjective or right regular, then $\textbf{fl}(ij_k)$ is the string that begins at $\tilde{i} = j - k$ and ends at $\tilde{j} = j \in FD$.
\end{itemize}
\end{defn} 

\noindent

\begin{rem}\label{rem: prepro and preinj strands defn}
Notice that with this convention, $ij_k$ is a preprojective string if and only if $\tilde{\varepsilon}_{\tilde{i}} = +$ and $\tilde{\varepsilon}_{\tilde{j}} = -$. Similarly, $ij_k$ is a preinjective string if and only if $\tilde{\varepsilon}_{\tilde{i}} = -$ and $\tilde{\varepsilon}_{\tilde{j}} = +$, $ij_k$ is a right regular string if and only if $\tilde{\varepsilon}_{\tilde{i}} = -=\tilde{\varepsilon}_{\tilde{j}}$ and $ij_k$ is a left regular string if and only if $\tilde{\varepsilon}_{\tilde{i}} = + = \tilde{\varepsilon}_{\tilde{j}}$. \\
\end{rem}

\begin{rem} \label{rem: counting dimensions of hom and ext in univ cover}
Let $M,N$ be two $\Bbbk Q^{\bm{\varepsilon}}$-string modules and fix $X \in F^{-1}(M)$. It follows from the definition of a covering functor that $\text{dim}(\text{Hom}_{\Bbbk Q}(M,N)) = \displaystyle \sum_{\{i:Y_i \in F^{-1}(N)\}} \text{dim}(\text{Hom}_{\Bbbk\tilde{Q}}(X,Y_i))$. \\

Moreover, although $\Bbbk \tilde{Q}^{\tilde{\bm{\varepsilon}}}$ is not hereditary, we have that for any $Y_i \in F^{-1}(N)$, $X$ and $Y_i$ are indecomposable string modules over some finite subquiver of $\tilde{Q}^{\tilde{\bm{\varepsilon}}}$ that is of type $\mathbb{A}_m$ for some $m\in \mathbb{N}$. Whence we can use the Auslander-Reiten formulas to conclude that 
\vspace{-.1 cm}
\begin{align*}
	\text{dim}(\text{Ext}_{\Bbbk Q}(M,N)) &= \text{dim}(\text{Hom}_{\Bbbk Q}(N,\tau M))  \\
											  &= \displaystyle \sum_{\{i:Y_i \in F^{-1}(N)\}} \text{dim}(\text{Hom}_{\Bbbk \tilde{Q}}(Y_i,X')) \\
											  &= \displaystyle \sum_{\{i:Y_i \in F^{-1}(N)\}} \text{dim}(\text{Ext}_{\Bbbk \tilde{Q}}(X,Y_i)).
\end{align*}

where $\tau X = X' \in F^{-1}(\tau M)$.

\end{rem}

\noindent
These remarks, along with the convention in Definition \ref{defn: fundamental lift} will become important in computing dimensions of extension groups and Hom spaces. 

\section{Combinatorial Realizations of Exceptional Collections}

In this section, we will provide three combinatorial realizations of exceptional collections in type $\tilde{\mathbb{A}}$; namely, strand diagrams, chord diagrams, and arc diagrams.

\subsection{Strand Diagrams}

\indent

In this subsection we will define a combinatorial object that we will call strand diagrams associated to string modules. We will then use these to classify exceptional collections over $Q^{\bm{\varepsilon}}$. Following [\ref{ref: combinatorics exceptional sequences type A}] closely, let $S_{n,\tilde{\varepsilon}} := \{\dots, (x_{-1},0), (x_0,0), (x_1,0), \dots\}\subset\mathbb{R}^2$ be a collection of points arranged on the $x$-axis from left to right together with the function $\tilde{\varepsilon} : S_{n,\tilde{\varepsilon}} \rightarrow \{+,-\}$ sending $(x_i,0) \mapsto \tilde{\varepsilon}_i\in\tilde{\bm\varepsilon}$. The reason for decorating this set with a subscript $n$ is because we will be taking these vertices modulo $n$ once we associate these diagrams to an $\tilde{\mathbb{A}}_{n-1}$ quiver. Moreover, we will typically label vertex $(x_i,0)$ as $i$.

\begin{defn}
Let $i,j\in\mathbb{Z}$ be such that $i\neq j$. A \textbf{strand} $c(i,j) = c(j,i)$ on $S_{n,\tilde{\varepsilon}}$ is an isotopy class of simple curves in $\mathbb{R}^2$ where any $\gamma \in c(i,j)$ satisfies: 
\begin{enumerate}
\item the endpoints of $\gamma$ are $(x_i,0)$ and $(x_j,0)$, 
\item as a subset of $\mathbb{R}^2$, $\gamma \subset \{(x,y)\in \mathbb{R}^2 : x_{\text{min}(i,j)} \leq x \leq x_{\text{max}(i,j)}\} \backslash \{(x_k,0) : x_{\text{min}(i,j)} < x_k < x_{\text{max}(i,j)}\}$, 
\item if $\text{min}(i,j) \leq k \leq  \text{max}(i,j)$ and $\tilde{\varepsilon}_k = +$ (resp., $\tilde{\varepsilon}_k = -$), then $\gamma$ is locally below (resp., locally above) $(x_k,0)$.
\end{enumerate}
\end{defn}  

By locally below (locally above)$(x_k,0)$ we mean that for a parameterization of $\gamma = (\gamma^{(1)}, \gamma^{(2)}):[0,1]\rightarrow \mathbb{R}^2$, there exists a $\delta \in \mathbb{R}$ with $0 < \delta < {1\over 2} \text{min}\{|x_k-x_{k-1}|,|x_k-x_{k+1}|\}$ such that $\gamma$ satisfies $\gamma^{(2)}(t) < 0$ if $\tilde{\varepsilon}_k = +$ (resp., $\gamma^{(2)}(t) >0 $ if $\tilde{\varepsilon}_k = -$) for all $t \in (0,1)$ where $\gamma^{(1)}(t) \in (x_k - \delta, x_k + \delta)$. Note that there is an injection $\Phi_{\tilde{\varepsilon}}$ from the string modules in ind(rep$(Q^{\tilde{\varepsilon}})$ to the set of strands on $S_{n,\tilde{\varepsilon}}$ given by  $\Phi_{\tilde{\varepsilon}}(ij_k) := c(\tilde{i},\tilde{j})$ where $\tilde{i} < \tilde{j}$. We call a strand $c(i,j) \in S_{n,\tilde{\varepsilon}}$ a \textbf{fundamental strand} if $c(i,j) \in \text{im}(\Phi_{\tilde{\varepsilon}})$. Note that any strand $c(i,j)$ can be represented by a \textbf{monotone curve} $\gamma \in c(i,j)$. That is, there exists a curve $\gamma \in c(i,j)$ with a parameterization $\gamma = (\gamma^{(1)},\gamma^{(2)}):[0,1] \rightarrow \mathbb{R}^2$ such that if $t,s\in [0,1]$ and $t<s$, then $\gamma^{(1)}(t) < \gamma^{(1)}(s)$.  

For the following definitions we fix some $n \in \mathbb{N}$, again keeping in mind that we will associate these strands to a quiver $Q^{\bm{\varepsilon}}$ where $|Q_0| = n$. We say that two strands $c(i_1,j_1)$ and $c(i_2,j_2)$ \textbf{intersect nontrivially} if there exists $z\in\mathbb{Z}$ such that any two curves $\gamma_1 \in c(i_{1} + nz,j_{1} + nz)$ and $\gamma_2 \in c(i_2,j_2)$ intersect in their interiors. Otherwise we say that $c(i_1,j_1)$ and $c(i_2,j_2)$ \textbf{do not intersect nontrivially}. We say a strand $c(i,j)$ \textbf{self-intersects} if there exists $z\in\mathbb{Z}$ such that any two curves $\gamma_1 \in c(i + nz,j + nz)$ and $\gamma_2 \in c(i,j)$ intersect in their interiors. Otherwise we say that $c(i,j)$ \textbf{does not self-intersect}. If $c(i_1,j_1)$ and $c(i_2,j_2)$ do not intersect nontrivially, we say $c(i_2,j_2)$ is \textbf{clockwise} from $c(i_1,j_1)$ (or equivalently $c(i_1,j_1)$ is \textbf{counterclockwise} from $c(i_2,j_2)$) if and only if for any $z \in \mathbb{Z}$ such that there exists $\gamma_1 \in c(i_{1} + zn,j_{1} + zn)$ and $\gamma_2 \in c(i_2,j_2)$ that share an endpoint $(x_k,0)$ and do not intersect in their interiors, we have that $\gamma_1$ and $\gamma_2$ locally appear in one of the six configurations Figure \ref{fig: six local clockwise configs}, preserving the property that $\gamma_1 \in c(i_{1} + zn,j_{1} + zn)$ and $\gamma_2 \in c(i_2,j_2)$. We say that $c(i_2,j_2)$ is \textbf{locally clockwise} from $c(i_1,j_1)$ if there exists $z \in \mathbb{Z}$ such that some $\gamma_1 \in c(i_{1} + zn,j_{1} + zn)$ and $\gamma_2 \in c(i_2,j_2)$ share an endpoint $(x_k,0)$, do not intersect in their interiors, and $\gamma_1$ and $\gamma_2$ locally appear in one of the six configurations in Figure \ref{fig: six local clockwise configs}, preserving the property that $\gamma_1 \in c(i_{1} + zn,j_{1} + zn)$ and $\gamma_2 \in c(i_2,j_2)$. We say that a collection of strands $\{c(i_1,j_1), c(i_2,j_2), \dots , c(i_k,j_k)\}$ form a \textbf{cycle} if and only if $c(i_l,j_l)$ is locally clockwise from $c(i_{l+1},j_{l+1})$ for all $l < k$ and $c(i_k,j_k)$ is locally clockwise from $c(i_1,j_1)$. We say a strand $c(i,j)$ is a \textbf{loop} if it does not self-intersect and $c(i,j)$ is both locally clockwise and locally counterclockwise from itself.

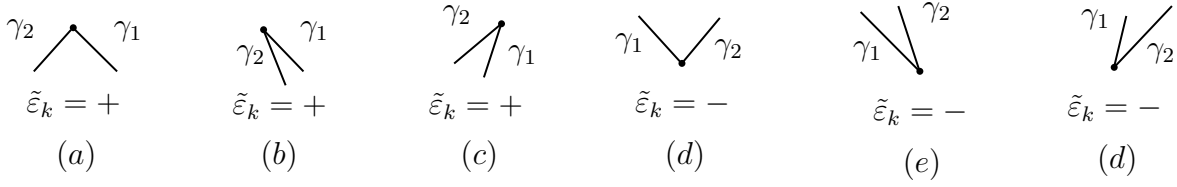
\begin{figure}[h!]
    \centering
\tikzset{every picture/.style={line width=0.75pt}} %set default line width to 0.75pt        

\begin{tikzpicture}[x=0.75pt,y=0.75pt,yscale=-1,xscale=1]
%uncomment if require: \path (0,332); %set diagram left start at 0, and has height of 332

%Straight Lines [id:da5125496577183359] 
\draw    (72,78) -- (94.22,100.01) ;
%Straight Lines [id:da1439658135047579] 
\draw    (71,79) -- (52.22,100.01) ;
%Straight Lines [id:da7574109717968718] 
\draw    (168,79) -- (188.22,100.01) ;
%Straight Lines [id:da7987429599382938] 
\draw    (168,79) -- (179.22,107.01) ;
%Straight Lines [id:da5785082571688516] 
\draw    (288,76) -- (279.22,103.01) ;
%Straight Lines [id:da33078244152019165] 
\draw    (288,76) -- (264.22,96.01) ;
%Straight Lines [id:da6310575474073974] 
\draw    (379,96) -- (398.22,73.01) ;
%Straight Lines [id:da7665146475502993] 
\draw    (379,96) -- (357.22,72.01) ;
%Straight Lines [id:da00007341384113845884] 
\draw    (499,100) -- (488.22,67.01) ;
%Straight Lines [id:da37347059621166356] 
\draw    (499,100) -- (469.22,70.01) ;
%Straight Lines [id:da10649586192391647] 
\draw    (597,98) -- (626.22,67.01) ;
%Straight Lines [id:da567161820238258] 
\draw    (597,98) -- (603.22,72.01) ;

% Text Node
\draw (66,70) node [anchor=north west][inner sep=0.75pt]  [font=\Huge]  {$\cdot $};
% Text Node
\draw (37,72.4) node [anchor=north west][inner sep=0.75pt]    {$\gamma _{2}$};
% Text Node
\draw (91,73.4) node [anchor=north west][inner sep=0.75pt]    {$\gamma _{1}$};
% Text Node
\draw (47,108.4) node [anchor=north west][inner sep=0.75pt]    {$\tilde{\varepsilon} _{k} =+$};
% Text Node
\draw (62,133.4) node [anchor=north west][inner sep=0.75pt]    {$( a)$};
% Text Node
\draw (162,71) node [anchor=north west][inner sep=0.75pt]  [font=\Huge]  {$\cdot $};
% Text Node
\draw (153,85.4) node [anchor=north west][inner sep=0.75pt]    {$\gamma _{2}$};
% Text Node
\draw (185,73.4) node [anchor=north west][inner sep=0.75pt]    {$\gamma _{1}$};
% Text Node
\draw (151,109.4) node [anchor=north west][inner sep=0.75pt]    {$\tilde{\varepsilon} _{k} =+$};
% Text Node
\draw (165,133.4) node [anchor=north west][inner sep=0.75pt]    {$( b)$};
% Text Node
\draw (282,68) node [anchor=north west][inner sep=0.75pt]  [font=\Huge]  {$\cdot $};
% Text Node
\draw (257,65.4) node [anchor=north west][inner sep=0.75pt]    {$\gamma _{2}$};
% Text Node
\draw (290,85.4) node [anchor=north west][inner sep=0.75pt]    {$\gamma _{1}$};
% Text Node
\draw (250,109.4) node [anchor=north west][inner sep=0.75pt]    {$\tilde{\varepsilon} _{k} =+$};
% Text Node
\draw (266,133.4) node [anchor=north west][inner sep=0.75pt]    {$( c)$};
% Text Node
\draw (373,88) node [anchor=north west][inner sep=0.75pt]  [font=\Huge]  {$\cdot $};
% Text Node
\draw (344,79.4) node [anchor=north west][inner sep=0.75pt]    {$\gamma _{1}$};
% Text Node
\draw (394,82.4) node [anchor=north west][inner sep=0.75pt]    {$\gamma _{2}$};
% Text Node
\draw (354,108.4) node [anchor=north west][inner sep=0.75pt]    {$\tilde{\varepsilon} _{k} =-$};
% Text Node
\draw (369,133.4) node [anchor=north west][inner sep=0.75pt]    {$( d)$};
% Text Node
\draw (493,92) node [anchor=north west][inner sep=0.75pt]  [font=\Huge]  {$\cdot $};
% Text Node
\draw (464,83.4) node [anchor=north west][inner sep=0.75pt]    {$\gamma _{1}$};
% Text Node
\draw (499,63.4) node [anchor=north west][inner sep=0.75pt]    {$\gamma _{2}$};
% Text Node
\draw (474,112.4) node [anchor=north west][inner sep=0.75pt]    {$\tilde{\varepsilon} _{k} =-$};
% Text Node
\draw (489,137.4) node [anchor=north west][inner sep=0.75pt]    {$( e)$};
% Text Node
\draw (591,90) node [anchor=north west][inner sep=0.75pt]  [font=\Huge]  {$\cdot $};
% Text Node
\draw (579,67.4) node [anchor=north west][inner sep=0.75pt]    {$\gamma _{1}$};
% Text Node
\draw (612,84.4) node [anchor=north west][inner sep=0.75pt]    {$\gamma _{2}$};
% Text Node
\draw (572,110.4) node [anchor=north west][inner sep=0.75pt]    {$\tilde{\varepsilon} _{k} =-$};
% Text Node
\draw (587,135.4) node [anchor=north west][inner sep=0.75pt]    {$( d)$}; 
\end{tikzpicture}

    \caption{The six possible local configurations of strand $c(i_2,j_2)$ being clockwise from strand $c(i_1,j_1)$ near the shared endpoint $(x_k,0)$.}
    \label{fig: six local clockwise configs}
\end{figure}

\noindent

A given collection of strands $d = \{c(i_l,j_l)\}_{[k]}$ with $k \leq n$, naturally defines a graph with vertex set $S_{n,\tilde{\varepsilon}}$ and edge set $\{((x_s,0),(x_t,0)) : c(s,t) \in d\}$. We refer to this graph as the \textbf{graph determined by d}. The following technical yet important lemma and its proof appear in [\ref{ref: combinatorics exceptional sequences type A}. Lemma 11]. 

\begin{lem} {\color{white} .}
\begin{enumerate}
\item Each strand $c(i,j)$ can be represented by a monotone curve $\gamma_{i,j}$.
\item Two distinct strands $c(i_1,j_1)$ and $c(i_2,j_2)$ on $S_{n,\tilde{\varepsilon}}$ that intersect nontrivially do not share an endpoint.
\item Two distinct strands $c(i_1,j_1)$ and $c(i_2,j_2)$ on $S_{n,\tilde{\varepsilon}}$ intersect nontrivially if and only if there exist representatives of the  monotone curves $\gamma_{i_1,j_2}$ and $\gamma_{i_2,j_2}$ that have a unique crossing. $\hfill \square$
\end{enumerate}
\end{lem}

\begin{defn}
A \textbf{fundamental strand diagram} $d = \{(c(i_l,j_l)\}_{l \in [n]}$ is a collection of $n$ strands on $S_{n,\tilde{\varepsilon}}$ that satisfies the following: 
\begin{enumerate}
\item $c(i_l,j_l)$ is a fundamental strand for all $l$,
\item $c(i_l,j_l)$ neither self-intersects nor forms a loop for all $l$,
\item distinct strands do not intersect nontrivially, and 
\item the graph determined by $d$ contains no cycles.
\end{enumerate}

\end{defn}

Let $D_{\tilde{\varepsilon}}$ denote the set of all fundamental strand diagrams on $S_{n,\tilde{\varepsilon}}$. The next lemma classifies when two nonisomorphic string modules of $Q^{\bm{\varepsilon}}$ define zero, one or two exceptional pairs. The statement of Lemma \ref{lem: key lemma} is nearly identical to that of Lemma 11 in [\ref{ref: combinatorics exceptional sequences type A}]. The proof of this statement using strand diagrams is however is more technical and tedious and relies on several other lemmas. To provide a shorter proof, we will reformulate this lemma using a different combinatorial object and prove it as Lemma \ref{lem: key lemma on annulus}.

\begin{lem} \label{lem: key lemma}
Let $Q^{\bm{\varepsilon}}$ be a quiver of type $\tilde{\mathbb{A}}_{n-1}$ and $U,V \in \text{ind}(\text{rep}_{\Bbbk}(Q^{\bm{\varepsilon}}))$ be two exceptional string modules. 

\begin{enumerate}
\item The strands $\Phi_{\tilde{\varepsilon}}(U)$ and $\Phi_{\tilde{\varepsilon}}(V)$ intersect nontrivially or form a cycle if and only if neither $(U,V)$ nor $(V,U)$ are exceptional pairs. 
\item The strand $\Phi_{\tilde{\varepsilon}}(U)$ is clockwise from $\Phi_{\tilde{\varepsilon}}(V)$ if and only if $(U,V)$ is an exceptional pair and $(V,U)$ is not. 
\item The strands $\Phi_{\tilde{\varepsilon}}(U)$ and $\Phi_{\tilde{\varepsilon}}(V)$ do not intersect at any of their endpoints (up to shift modulo $n$) and they do not intersect nontrivially if and only if both $(U,V)$ and $(V,U)$ form exceptional pairs. 
\end{enumerate}

%Moreover, there exist monotone curves $\gamma_{\varepsilon}(U) \in \Phi_{\varepsilon}(U)$ for all $U \in \text{ind}(\text{rep}_{\Bbbk}(Q^{\bm{\varepsilon}}))$ so that for any modulo $n$ shift of $\gamma_{\varepsilon}(U)$, we have that $\gamma_{\varepsilon}(U)$ and $\gamma_{\varepsilon}(V)$ have a unique crossing, have a common endpoint, or have no point of intersection, respectively, if and only if $U$ and $V$ satisfy 1, 2 or 3 respectively. NOT SURE THIS HOLDS. CAN HAVE A COMMON ENDPOINT AND 2. NOT HOLD
\end{lem}

Our first result follows from Lemma \ref{lem: key lemma}. The proof of Theorem \ref{thm: bijection with strand diagrams} is nearly identical to that of Theorem 12 in [\ref{ref: combinatorics exceptional sequences type A}] and we will reformulate the statement and prove it as Theorem \ref{thm: ec bijection with arc diagrams}. 

\begin{thm}\label{thm: bijection with strand diagrams}
Let $\bar{E}_{\bm{\varepsilon}} := \{ \text{exceptional collections of } Q^{\bm{\varepsilon}}\}$ where $Q^{\bm{\varepsilon}}$ is a quiver of type $\tilde{\mathbb{A}}_{n-1}$. There is a bijection $\bar{E}_{\bm{\varepsilon}} \overset{\bar{\Phi}_{\bm{\varepsilon}}}{\longrightarrow} D_{\tilde{\varepsilon}}$ given by $\bar{\Phi}_{\bm{\varepsilon}}: \{i_lj_{l_k}\}_{l \in [n]} \mapsto \{c(\tilde{i}_l,\tilde{j}_l)\}_{l \in [n]}$. $\hfill \square$ \\
\end{thm} 

\begin{exmp} \label{exmp: example of strand diagram corresp to exceptional collection}
Consider the quiver $Q^{\bm{\varepsilon}}$ with $\bm{\varepsilon} = (-,+,-,+,+)$ as in Figure \ref{fig: underlying graph and example of orientation}. Then we have an exceptional collection given by $\{13_2, 14_3,15_4, 25_3, 12_6\}$. The corresponding fundamental strand diagram is shown below with the black strands corresponding to the regular modules and the other strands corresponding to the preprojectives.

\begin{center}

\tikzset{every picture/.style={line width=0.75pt}} %set default line width to 0.75pt        

\begin{tikzpicture}[x=0.75pt,y=0.75pt,yscale=-.9,xscale=1]
%uncomment if require: \path (0,300); %set diagram left start at 0, and has height of 332

%Curve Lines [id:da9195622695030736] 
\draw    (248,127) .. controls (249.94,160.24) and (280.01,37.83) .. (305.11,70.42) .. controls (330.22,103.01) and (342.22,148.01) .. (346.22,125.01) ;
%Curve Lines [id:da9660419618718619] 
\draw    (248,127) .. controls (249.13,173.27) and (277.46,44.9) .. (307.04,86.42) .. controls (336.61,127.95) and (332.22,159.01) .. (405.22,126.01) ;
%Curve Lines [id:da9760255838395631] 
\draw  [color={rgb, 255:red, 0; green, 0; blue, 255 }  ]   (248,127) .. controls (247.31,175.59) and (290.22,58.01) .. (304.33,89.81) .. controls (318.44,121.62) and (340.22,168.01) .. (396.68,138.56) .. controls (453.13,109.11) and (454.76,80.08) .. (466.22,81.01) .. controls (477.68,81.94) and (498.36,141.2) .. (525.3,133.26) .. controls (552.23,125.32) and (560.22,78.01) .. (570.22,111.01) ;
%Curve Lines [id:da7305230628629944] 
\draw   [color={rgb, 255:red, 0; green, 255; blue, 0 }  ]  (248,127) .. controls (247.3,176.66) and (293.23,67.67) .. (305.23,99.34) .. controls (317.22,131.01) and (360.59,186.79) .. (415.41,137.4) .. controls (470.22,88.01) and (454.22,90.01) .. (466.22,112.01) ;
%Curve Lines [id:da718535270920533] 
\draw   (298,112) .. controls (298.2,78.74) and (339.22,201.01) .. (399.84,160.2) .. controls (460.46,119.39) and (447.22,91.01) .. (466.22,112.01) ;

% Text Node
\draw (241,109.4) node [anchor=north west][inner sep=0.75pt]    {$1$};
% Text Node
\draw (293,109.4) node [anchor=north west][inner sep=0.75pt]    {$2$};
% Text Node
\draw (341,109.4) node [anchor=north west][inner sep=0.75pt]    {$3$};
% Text Node
\draw (401,109.4) node [anchor=north west][inner sep=0.75pt]    {$4$};
% Text Node
\draw (460,110.4) node [anchor=north west][inner sep=0.75pt]    {$5$};
% Text Node
\draw (514,110.4) node [anchor=north west][inner sep=0.75pt]    {$6$};
% Text Node
\draw (565,111.4) node [anchor=north west][inner sep=0.75pt]    {$7$};
% Text Node
\draw (192,109.4) node [anchor=north west][inner sep=0.75pt]    {$0$};

% Text Node
\draw (600,115.4) node [anchor=north west][inner sep=0.75pt]    {$\dotsc $};
% Text Node
\draw (155,115.4) node [anchor=north west][inner sep=0.75pt]    {$\dotsc $};

\end{tikzpicture}

\end{center}
\end{exmp}
\vspace{-1.5 cm}

\subsection{Chord Diagrams and Arc Diagrams} \label{sec: chords and annuli}

Let $S_{n,+} := \{\dots, (1,y_{-1}), (1,y_0), (1,y_1), \dots\}\subset\mathbb{R}^2$ and $S_{n,-} := \{\dots, (-1,y'_{-1}), (-1,y'_0)$, 
$(-1,y'_1), \dots\}\subset\mathbb{R}^2$ be two collections of points such that $y_i \neq y'_j$ for all $i$ and $j$, arranged on the lines $x = 1$ and $x = -1$ respectively from bottom to top in natural numerical order together with the function $\tilde{\varepsilon} : S_{n,+(-)} \rightarrow \{+(-)\}$ sending $(1(-1),y_i(y'_i)) \mapsto \tilde{\varepsilon_i}$. Again, the reason for decorating this set with a subscript $n$ is because we will be taking these vertices modulo $n$ once we associate these diagrams to an $\tilde{\mathbb{A}}_{n-1}$ quiver.

\begin{defn} 
A \textbf{chord} $C(i,j)=C(j,i)$ is defined as follows:  
\begin{itemize}
\item if $y_i \in S_{n,+(-)}$ and $y'_j \in S_{n,-(+)}$ then $C(i,j)$ is the straight line segment beginning at  $(1(-1),y_i)$ and ending at $(-1(1),y'_j)$.
\item if $y_i \in S_{n,+(-)}$ and $y_j \in S_{n,+(-)}$ then $C(i,j)$ is an isotopy class of simple curves in $\mathbb{R}^2$ where any $\gamma \in C(i,j)$ satisfies: 
\begin{enumerate}
\item the endpoints of $\gamma$ are $(1(-1),y_i)$ and $(1(-1),y_j)$, 
\item as a subset of $\mathbb{R}^2$, $\gamma \subset \{(x,y)\in \mathbb{R}^2 : y_{\text{min}(i,j)} \leq y \leq y_{\text{max}(i,j)}, \, -1 < x < 1\} \cup \{(1(-1),y_i),(1(-1),y_j)\}$.
\end{enumerate}
\end{itemize} 
\end{defn}

Note that any non-linear chord $C(i,j)$ can be represented by a \textbf{monotone curve} $\gamma \in C(i,j)$. A \textbf{chord diagram} is any collection of chords on $S_{n,+} \cup S_{n,-}$. The \textbf{fundamental domain} of a chord diagram is $\{(x,y) : -1 \leq x \leq 1, \, \text{min}(y_0,y'_0) < y < \text{max}(y_n,y'_n)\}$. Any other fundamental domain is given by shifting the fundamental domain vertically modulo $n$. Given a string module, we define the \textbf{fundamental chord} associated to the module analogously to Definition \ref{defn: fundamental lift}. We say that two chords $C(i_1,j_1)$ and $C(i_2,j_2)$ \textbf{intersect nontrivially} if there exists $z\in\mathbb{Z}$ such that any two representatives $\gamma_1 \in C(i_{1} + nz,j_{1} + nz)$ and $\gamma_2 \in C(i_2,j_2)$ intersect in their interiors. Note that we can always take this intersection to occur in the interior of a fundamental domain. Otherwise we say that $C(i_1,j_1)$ and $C(i_2,j_2)$ \textbf{do not intersect nontrivially}. We say a chord $C(i,j)$ \textbf{self-intersects} if there exists $z\in\mathbb{Z}$ such that any two representatives $\gamma_1 \in C(i + nz,j + nz)$ and $\gamma_2 \in C(i,j)$ intersect in their interiors. Otherwise we say that $C(i,j)$ \textbf{does not self-intersect}. If $C(i_1,j_1)$ and $C(i_2,j_2)$ do not intersect nontrivially, we say $C(i_2,j_2)$ is \textbf{clockwise} from $C(i_1,j_1)$ (or equivalently $C(i_1,j_1)$ is \textbf{counterclockwise} from $C(i_2,j_2)$) if and only if for all $z \in \mathbb{Z}$ such that $C(i_1+nz,j_1+nz)$ and $C(i_2,j_2)$ share an endpoint, we have the following. If this point is in $S_{n,+}\, (S_{n,-})$, then the chord lying locally above (below) the other is clockwise from the lower (upper) chord. We say $C(i_2,j_2)$ is \textbf{locally clockwise} from $C(i_1,j_1)$ (or equivalently $C(i_1,j_1)$ is \textbf{locally counterclockwise} from $C(i_2,j_2)$) if there exists such a $z$. As before, we say that a collection of chords $\{C(i_1,j_1), C(i_2,j_2), \dots , C(i_k,j_k)\}$ form a \textbf{cycle} if and only if $C(i_l,j_l)$ is locally clockwise from $C(i_{l+1},j_{l+1})$ for all $l < k$ and $C(i_k,j_k)$ is locally clockwise from $C(i_1,j_1)$. We say a chord $C(i,j)$ is a \textbf{loop} if it does not self-intersect and $C(i,j)$ is both locally clockwise and locally counterclockwise from itself.

\begin{defn}
A \textbf{fundamental chord diagram} $C_{Q^{\bm{\varepsilon}}}$ is a collection of $n$ chords on $S_{n,+} \cup S_{n,-}$ that satisfies the following: 
\begin{enumerate}
\item all $n$ chords are fundamental,
\item all chords neither self-intersect nor form a loop,
\item distinct chords do not intersect nontrivially, and 
\item the chords do not form any cycles.
\end{enumerate}

\end{defn}

Let $C_{\tilde{\varepsilon}}$ denote the set of fundamental chord diagrams $C_{Q^{\bm{\varepsilon}}}$. Note that there exists a homeomorphism between any two fundamental strand and chord diagrams. Given a strand diagram, if $\tilde{\varepsilon}_i = +(-)$, we send $(x_i,0)$ to $(x_i,1(-1))$, then straighten the strands that begin (end) at $(x_i,1(-1))$ and end (begin) at $(x_j,-1(1))$. We then send those that begin at $(x_i,1(-1))$ and end at $(x_j,1(-1))$ to a simple monotone curve. Finally a ${\pi\over 2}$ counterclockwise rotation about the origin gives the fundamental chord diagram, as in the next Example. Call this homeomorphism $\Psi$.

\begin{exmp} \label{exmp: fundamental chord diagram}

For $Q^{\bm{\varepsilon}}$ where $\bm{\varepsilon} = (-, +, +, -)$, $\Psi$ acts as follows.

\begin{center}
\tikzset{every picture/.style={line width=0.75pt}} %set default line width to 0.75pt        

\begin{tikzpicture}[x=0.65pt,y=0.65pt,yscale=-1,xscale=1]
%uncomment if require: \path (0,332); %set diagram left start at 0, and has height of 332

%Curve Lines [id:da3036210712020331] 
\draw    [color={rgb, 255:red, 189; green, 16; blue, 224 }  ,draw opacity=1 ] (80,159) .. controls (88.22,129.01) and (108.22,212.01) .. (131.22,179.01) ;
%Curve Lines [id:da21465303720885132] 
\draw  [color={rgb, 255:red, 208; green, 2; blue, 27 }  ,draw opacity=1 ]  (131.22,179.01) .. controls (196.22,236.01) and (223.22,123.01) .. (235.22,158.01) ;
%Curve Lines [id:da47960950415410575] 
\draw   [color={rgb, 255:red, 74; green, 144; blue, 226 }  ,draw opacity=1 ]   (181,179) .. controls (186.22,206.01) and (230.22,122.01) .. (235.22,158.01) ;
%Straight Lines [id:da7193263980688573] 
\draw    (237,130) -- (306.55,84.11) ;
\draw [shift={(308.22,83.01)}, rotate = 146.58] [color={rgb, 255:red, 0; green, 0; blue, 0 }  ][line width=0.75]    (10.93,-3.29) .. controls (6.95,-1.4) and (3.31,-0.3) .. (0,0) .. controls (3.31,0.3) and (6.95,1.4) .. (10.93,3.29)   ;
%Curve Lines [id:da8722491841847759] 
\draw   [color={rgb, 255:red, 248; green, 231; blue, 28 }  ,draw opacity=1 ]   (32.22,161.01) .. controls (95.22,95.01) and (107.22,206.01) .. (131.22,179.01) ;
%Curve Lines [id:da1035908110842576] 
\draw   [color={rgb, 255:red, 189; green, 16; blue, 224 }  ,draw opacity=1 ] (395.22,93.01) .. controls (403.44,63.01) and (426.22,68.01) .. (449.22,35.01) ;
%Curve Lines [id:da9527740521079862] 
\draw     [color={rgb, 255:red, 208; green, 2; blue, 27 }  ,draw opacity=1 ] (449.22,35.01) .. controls (514.22,92.01) and (539.22,63.01) .. (551.22,98.01) ;
%Curve Lines [id:da8049551432415265] 
\draw    [color={rgb, 255:red, 74; green, 144; blue, 226 }  ,draw opacity=1 ]  (499.22,33.01) .. controls (504.44,60.01) and (546.22,62.01) .. (551.22,98.01) ;
%Curve Lines [id:da32799415945186383] 
\draw   [color={rgb, 255:red, 248; green, 231; blue, 28 }  ,draw opacity=1 ]  (346.22,91.01) .. controls (409.22,25.01) and (425.22,62.01) .. (449.22,35.01) ;
%Straight Lines [id:da31224183258514215] 
\draw    (455,91) -- (457.17,175.01) ;
\draw [shift={(457.22,177.01)}, rotate = 268.52] [color={rgb, 255:red, 0; green, 0; blue, 0 }  ][line width=0.75]    (10.93,-3.29) .. controls (6.95,-1.4) and (3.31,-0.3) .. (0,0) .. controls (3.31,0.3) and (6.95,1.4) .. (10.93,3.29)   ;
%Straight Lines [id:da1474305651787864] 
\draw  [color={rgb, 255:red, 74; green, 144; blue, 226 }  ,draw opacity=1 ]  (400,220) -- (529.22,203.01) ;
%Straight Lines [id:da3295732874179118] 
\draw   [color={rgb, 255:red, 208; green, 2; blue, 27 }  ,draw opacity=1 ]   (400.22,254.01) -- (529.22,203.01) ;
%Straight Lines [id:da12460306497433704] 
\draw  [color={rgb, 255:red, 189; green, 16; blue, 224 }  ,draw opacity=1 ]  (400.22,254.01) -- (533.22,275.01) ;
%Straight Lines [id:da18639619435095778] 
\draw  [color={rgb, 255:red, 248; green, 231; blue, 28 }  ,draw opacity=1 ]   (400.22,254.01) -- (523.22,313.01) ;

% Text Node
\draw (128,161.4) node [anchor=north west][inner sep=0.75pt]    {$1$};
% Text Node
\draw (74,161.4) node [anchor=north west][inner sep=0.75pt]    {$0$};
% Text Node
\draw (178,161.4) node [anchor=north west][inner sep=0.75pt]    {$2$};
% Text Node
\draw (230,161.4) node [anchor=north west][inner sep=0.75pt]    {$3$};
% Text Node
\draw (16,162.4) node [anchor=north west][inner sep=0.75pt]    {$-1$};
% Text Node
\draw (271,160.4) node [anchor=north west][inner sep=0.75pt]    {$4$};
% Text Node
\draw (443,15.4) node [anchor=north west][inner sep=0.75pt]    {$1$};
% Text Node
\draw (389,95.4) node [anchor=north west][inner sep=0.75pt]    {$0$};
% Text Node
\draw (493,14.4) node [anchor=north west][inner sep=0.75pt]    {$2$};
% Text Node
\draw (546,98.4) node [anchor=north west][inner sep=0.75pt]    {$3$};
% Text Node
\draw (330,94.4) node [anchor=north west][inner sep=0.75pt]    {$-1$};
% Text Node
\draw (586,99.4) node [anchor=north west][inner sep=0.75pt]    {$4$};
% Text Node
\draw (389,246.4) node [anchor=north west][inner sep=0.75pt]    {$1$};
% Text Node
\draw (534,266.4) node [anchor=north west][inner sep=0.75pt]    {$0$};
% Text Node
\draw (388,212.4) node [anchor=north west][inner sep=0.75pt]    {$2$};
% Text Node
\draw (530,192.4) node [anchor=north west][inner sep=0.75pt]    {$3$};
% Text Node
\draw (526,309.4) node [anchor=north west][inner sep=0.75pt]    {$-1$};
% Text Node
\draw (531,159.4) node [anchor=north west][inner sep=0.75pt]    {$4$};

\end{tikzpicture}

\end{center}

\end{exmp}

We can conclude that two chords cross if and only if their preimages under $\Psi$ cross. Moreover, we notice that $\Psi$ preserves the clockwise nature of strands. Therefore we have the following reformulation of Lemma \ref{lem: key lemma}: 

\begin{lem}
Let $Q^{\bm{\varepsilon}}$ be a quiver of type $\tilde{\mathbb{A}}_{n-1}$ and let $U,V \in \text{ind}(\text{rep}_{\Bbbk}(Q^{\bm{\varepsilon}}))$ be two exceptional string modules. 

\begin{enumerate}
\item The fundamental chords $\Psi \circ \Phi_{\tilde{\varepsilon}}(U)$ and $\Psi \circ \Phi_{\tilde{\varepsilon}}(V)$ intersect nontrivially or form a cycle if and only if neither $(U,V)$ nor $(V,U)$ are exceptional pairs. 
\item The chord $\Psi \circ \Phi_{\tilde{\varepsilon}}(U)$ is clockwise from $\Psi \circ \Phi_{\tilde{\varepsilon}}(V)$ if and only if $(U,V)$ is an exceptional pair and $(V,U)$ is not. 
\item The chords $\Psi \circ \Phi_{\tilde{\varepsilon}}(U)$ and $\Psi \circ \Phi_{\tilde{\varepsilon}}(V)$ do not intersect at any of their endpoints (up to shift modulo $n$) and they do not intersect nontrivially if and only if both $(U,V)$ and $(V,U)$ form exceptional pairs.  \qed
\end{enumerate}
\end{lem}

\noindent
As a consequence, we have the following reformulation of Theorem \ref{thm: bijection with strand diagrams}.

\begin{thm}
Let $\bar{E}_{\tilde{\varepsilon}} := \{ \text{exceptional collections of } Q^{\bm{\varepsilon}}\}$ where $Q^{\bm{\varepsilon}}$ is a quiver of type $\tilde{\mathbb{A}}_{n-1}$. There is a bijection $\bar{E}_{\tilde{\varepsilon}} \overset{\overline{\Psi \circ \Phi}_{\tilde{\varepsilon}}}{\longrightarrow} C_{\tilde{\varepsilon}}$ given by $\overline{\Psi \circ \Phi}_{\tilde{\varepsilon}}: \{i_lj_{l_k}\}_{l \in [n]} \mapsto \{C(\tilde{i}_l,\tilde{j}_l)\}_{l \in [n]}$. $\hfill \square$ 
\end{thm}

Given our quiver $Q^{\bm{\varepsilon}}$, we will now associate an annulus $A_{Q^{\bm{\varepsilon}}}$ whose difference between outer and inner radii is $r$ as follows. If $\varepsilon_i = +(-)$, then $i$ is a marked point  on the outer (inner) circle of the annulus. We moreover write the vertices in clockwise order respecting the natural numerical order of the vertices. We adopt the convention that $i \in \{0,1,\dots,n-1\}$ where we identify $n$ and $0$. 

\begin{defn}
Let $i,j\in\{0,1,\dots , n-1\}$ be such that $i\neq j$. An \textbf{arc} $a(i,j)[\lambda]$ on $A_{Q^{\bm{\varepsilon}}}$ is an isotopy class of simple curves in $A_{Q^{\bm{\varepsilon}}}$ where any $\gamma \in a(i,j)[\lambda]$ satisfies: 
\begin{enumerate}
\item $\gamma$ begins at $i$ and ends at $j$, 
\item $\gamma$ travels clockwise through the interior of the annulus from $i$ to $j$. 
\item If $\gamma$ begins and ends on the same boundary component, the integer $\lambda$ is the winding number of $\gamma$ about the inner boundary component. If $\gamma$ connects the two boundary components, $\lambda$ is the clockwise winding number of $\gamma$ about the inner boundary circle of $A_{Q^{\bm{\varepsilon}}}$ when traversing $\gamma$ beginning from the outer boundary component. 
\end{enumerate}
  
A collection of such arcs will be called an \textbf{arc diagram}. If an arc from $i$ to $j$ has counter clockwise winding number $k$, we write $a(i,j)[-k]$. Arcs that begin at one boundary component and end at another will be called \textbf{bridging arcs}, while those that begin and end on the same boundary component will be called \textbf{exterior arcs}.  

\end{defn}

\begin{rem}
This annulus is not new in the study of indecomposable representations of type $\tilde{\mathbb{A}}$ quivers. In fact, mathematicians have been using triangulations of marked surfaces to study these algebras as seen in [\ref{ref: clusters are in bijection with triangulations}], [\ref{ref: Master's Thesis Clusters and Triangulations}], and [\ref{ref: BCS}] to name a few. The main difference between the model presented here and the models in the aforementioned papers is that the bijection between arcs and indecomposable representations is different. In the aforementioned papers, the algebra is given by a triangulation of the marked surface. Upon drawing an arc in this triangulation, this arc crosses a finite number of the diagonals in the triangulation. The number of times the arc crosses a particular diagonal then gives the dimension of the vectorspace at the vertex corresponding to this arc. Here, our bijection between arcs and indecomposable representations is different. In [\ref{ref: Igusa Maresca}], Igusa and the author have shown precisely when the convention in [\ref{ref: Master's Thesis Clusters and Triangulations}] and the one in this paper coincide.
\end{rem}

\begin{rem}\label{rem: Duality on Annulus}
It is well known that there is a duality $D: \text{rep}_{\Bbbk}Q \rightarrow \text{rep}_{\Bbbk}Q^{op}$. Suppose the quiver $Q^{\bm\varepsilon}$ has annulus $A_{Q^{\bm{\varepsilon}}}$ with $p$ marked points on the inner boundary and $q$ on the outer boundary. Then the annulus associated to the opposite quiver is the one with $q$ marked points on the inner boundary and $p$ on the outer boundary. We can visualize the duality functor topologically as a homeomorphism as follows. 
\begin{center}
\resizebox{12cm}{9cm}{%

\tikzset{every picture/.style={line width=0.75pt}} %set default line width to 0.75pt        

\begin{tikzpicture}[x=0.75pt,y=0.75pt,yscale=-1,xscale=1]
%uncomment if require: \path (0,504); %set diagram left start at 0, and has height of 504

%Shape: Donut [id:dp44169344237469477] 
\draw   (103.35,411.7) .. controls (103.35,396.95) and (114.69,384.99) .. (128.68,384.99) .. controls (142.66,384.99) and (154.01,396.95) .. (154.01,411.7) .. controls (154.01,426.44) and (142.66,438.4) .. (128.68,438.4) .. controls (114.69,438.4) and (103.35,426.44) .. (103.35,411.7)(65.35,411.7) .. controls (65.35,375.97) and (93.7,347) .. (128.68,347) .. controls (163.65,347) and (192,375.97) .. (192,411.7) .. controls (192,447.43) and (163.65,476.39) .. (128.68,476.39) .. controls (93.7,476.39) and (65.35,447.43) .. (65.35,411.7) ;
%Straight Lines [id:da2832101627824686] 
\draw    (128.68,384.99) ;
\draw [shift={(128.68,384.99)}, rotate = 0] [color={rgb, 255:red, 0; green, 0; blue, 0 }  ][fill={rgb, 255:red, 0; green, 0; blue, 0 }  ][line width=0.75]      (0, 0) circle [x radius= 3.35, y radius= 3.35]   ;
%Straight Lines [id:da8764043787531499] 
\draw    (154.01,411.7) ;
\draw [shift={(154.01,411.7)}, rotate = 0] [color={rgb, 255:red, 0; green, 0; blue, 0 }  ][fill={rgb, 255:red, 0; green, 0; blue, 0 }  ][line width=0.75]      (0, 0) circle [x radius= 3.35, y radius= 3.35]   ;
%Straight Lines [id:da6740203285552393] 
\draw    (103.35,411.7) ;
\draw [shift={(103.35,411.7)}, rotate = 0] [color={rgb, 255:red, 0; green, 0; blue, 0 }  ][fill={rgb, 255:red, 0; green, 0; blue, 0 }  ][line width=0.75]      (0, 0) circle [x radius= 3.35, y radius= 3.35]   ;
%Straight Lines [id:da03935715887250302] 
\draw    (169.01,362.7) ;
\draw [shift={(169.01,362.7)}, rotate = 0] [color={rgb, 255:red, 0; green, 0; blue, 0 }  ][fill={rgb, 255:red, 0; green, 0; blue, 0 }  ][line width=0.75]      (0, 0) circle [x radius= 3.35, y radius= 3.35]   ;
%Straight Lines [id:da47864981403755746] 
\draw    (167.35,462.39) ;
\draw [shift={(167.35,462.39)}, rotate = 0] [color={rgb, 255:red, 0; green, 0; blue, 0 }  ][fill={rgb, 255:red, 0; green, 0; blue, 0 }  ][line width=0.75]      (0, 0) circle [x radius= 3.35, y radius= 3.35]   ;
%Straight Lines [id:da916675132570649] 
\draw    (75.35,447.39) ;
\draw [shift={(75.35,447.39)}, rotate = 0] [color={rgb, 255:red, 0; green, 0; blue, 0 }  ][fill={rgb, 255:red, 0; green, 0; blue, 0 }  ][line width=0.75]      (0, 0) circle [x radius= 3.35, y radius= 3.35]   ;
%Straight Lines [id:da11915959199725412] 
\draw    (82.35,368.39) ;
\draw [shift={(82.35,368.39)}, rotate = 0] [color={rgb, 255:red, 0; green, 0; blue, 0 }  ][fill={rgb, 255:red, 0; green, 0; blue, 0 }  ][line width=0.75]      (0, 0) circle [x radius= 3.35, y radius= 3.35]   ;
%Shape: Donut [id:dp408510164407041] 
\draw   (489.35,410.7) .. controls (489.35,395.95) and (500.69,383.99) .. (514.68,383.99) .. controls (528.66,383.99) and (540.01,395.95) .. (540.01,410.7) .. controls (540.01,425.44) and (528.66,437.4) .. (514.68,437.4) .. controls (500.69,437.4) and (489.35,425.44) .. (489.35,410.7)(451.35,410.7) .. controls (451.35,374.97) and (479.7,346) .. (514.68,346) .. controls (549.65,346) and (578,374.97) .. (578,410.7) .. controls (578,446.43) and (549.65,475.39) .. (514.68,475.39) .. controls (479.7,475.39) and (451.35,446.43) .. (451.35,410.7) ;
%Straight Lines [id:da7747121428928458] 
\draw    (532.35,390.39) ;
\draw [shift={(532.35,390.39)}, rotate = 0] [color={rgb, 255:red, 0; green, 0; blue, 0 }  ][fill={rgb, 255:red, 0; green, 0; blue, 0 }  ][line width=0.75]      (0, 0) circle [x radius= 3.35, y radius= 3.35]   ;
%Straight Lines [id:da18089562679402515] 
\draw    (533.01,429.7) ;
\draw [shift={(533.01,429.7)}, rotate = 0] [color={rgb, 255:red, 0; green, 0; blue, 0 }  ][fill={rgb, 255:red, 0; green, 0; blue, 0 }  ][line width=0.75]      (0, 0) circle [x radius= 3.35, y radius= 3.35]   ;
%Straight Lines [id:da5778902154057519] 
\draw    (493.35,424.7) ;
\draw [shift={(493.35,424.7)}, rotate = 0] [color={rgb, 255:red, 0; green, 0; blue, 0 }  ][fill={rgb, 255:red, 0; green, 0; blue, 0 }  ][line width=0.75]      (0, 0) circle [x radius= 3.35, y radius= 3.35]   ;
%Straight Lines [id:da24340242330879414] 
\draw    (514.68,346) ;
\draw [shift={(514.68,346)}, rotate = 0] [color={rgb, 255:red, 0; green, 0; blue, 0 }  ][fill={rgb, 255:red, 0; green, 0; blue, 0 }  ][line width=0.75]      (0, 0) circle [x radius= 3.35, y radius= 3.35]   ;
%Straight Lines [id:da2956020711349574] 
\draw    (578,410.7) ;
\draw [shift={(578,410.7)}, rotate = 0] [color={rgb, 255:red, 0; green, 0; blue, 0 }  ][fill={rgb, 255:red, 0; green, 0; blue, 0 }  ][line width=0.75]      (0, 0) circle [x radius= 3.35, y radius= 3.35]   ;
%Straight Lines [id:da22099450856363978] 
\draw    (451.35,410.7) ;
\draw [shift={(451.35,410.7)}, rotate = 0] [color={rgb, 255:red, 0; green, 0; blue, 0 }  ][fill={rgb, 255:red, 0; green, 0; blue, 0 }  ][line width=0.75]      (0, 0) circle [x radius= 3.35, y radius= 3.35]   ;
%Straight Lines [id:da13218425471192252] 
\draw    (498.35,390.39) ;
\draw [shift={(498.35,390.39)}, rotate = 0] [color={rgb, 255:red, 0; green, 0; blue, 0 }  ][fill={rgb, 255:red, 0; green, 0; blue, 0 }  ][line width=0.75]      (0, 0) circle [x radius= 3.35, y radius= 3.35]   ;
%Shape: Circle [id:dp9413081247002455] 
\draw   (68.61,203.2) .. controls (68.61,173.82) and (92.42,150) .. (121.8,150) .. controls (151.18,150) and (175,173.82) .. (175,203.2) .. controls (175,232.58) and (151.18,256.39) .. (121.8,256.39) .. controls (92.42,256.39) and (68.61,232.58) .. (68.61,203.2) -- cycle ;
%Shape: Circle [id:dp00823892041990737] 
\draw   (88.61,53.2) .. controls (88.61,34.31) and (103.92,19) .. (122.8,19) .. controls (141.69,19) and (157,34.31) .. (157,53.2) .. controls (157,72.08) and (141.69,87.39) .. (122.8,87.39) .. controls (103.92,87.39) and (88.61,72.08) .. (88.61,53.2) -- cycle ;
%Straight Lines [id:da8435365132822674] 
\draw    (157,53.2) -- (175,203.2) ;
%Straight Lines [id:da5944260108955193] 
\draw    (88.61,53.2) -- (68.61,203.2) ;
%Straight Lines [id:da18639760976018027] 
\draw    (89.35,160.39) ;
\draw [shift={(89.35,160.39)}, rotate = 0] [color={rgb, 255:red, 0; green, 0; blue, 0 }  ][fill={rgb, 255:red, 0; green, 0; blue, 0 }  ][line width=0.75]      (0, 0) circle [x radius= 3.35, y radius= 3.35]   ;
%Straight Lines [id:da12027674728277438] 
\draw    (154.35,161.39) ;
\draw [shift={(154.35,161.39)}, rotate = 0] [color={rgb, 255:red, 0; green, 0; blue, 0 }  ][fill={rgb, 255:red, 0; green, 0; blue, 0 }  ][line width=0.75]      (0, 0) circle [x radius= 3.35, y radius= 3.35]   ;
%Straight Lines [id:da9581575882005804] 
\draw    (156.35,243.39) ;
\draw [shift={(156.35,243.39)}, rotate = 0] [color={rgb, 255:red, 0; green, 0; blue, 0 }  ][fill={rgb, 255:red, 0; green, 0; blue, 0 }  ][line width=0.75]      (0, 0) circle [x radius= 3.35, y radius= 3.35]   ;
%Straight Lines [id:da8001757258072237] 
\draw    (73.35,225.39) ;
\draw [shift={(73.35,225.39)}, rotate = 0] [color={rgb, 255:red, 0; green, 0; blue, 0 }  ][fill={rgb, 255:red, 0; green, 0; blue, 0 }  ][line width=0.75]      (0, 0) circle [x radius= 3.35, y radius= 3.35]   ;
%Straight Lines [id:da37023992045113663] 
\draw    (122.8,19) ;
\draw [shift={(122.8,19)}, rotate = 0] [color={rgb, 255:red, 0; green, 0; blue, 0 }  ][fill={rgb, 255:red, 0; green, 0; blue, 0 }  ][line width=0.75]      (0, 0) circle [x radius= 3.35, y radius= 3.35]   ;
%Straight Lines [id:da17047510455131598] 
\draw    (157,53.2) ;
\draw [shift={(157,53.2)}, rotate = 0] [color={rgb, 255:red, 0; green, 0; blue, 0 }  ][fill={rgb, 255:red, 0; green, 0; blue, 0 }  ][line width=0.75]      (0, 0) circle [x radius= 3.35, y radius= 3.35]   ;
%Straight Lines [id:da39384103267364257] 
\draw    (88.61,53.2) ;
\draw [shift={(88.61,53.2)}, rotate = 0] [color={rgb, 255:red, 0; green, 0; blue, 0 }  ][fill={rgb, 255:red, 0; green, 0; blue, 0 }  ][line width=0.75]      (0, 0) circle [x radius= 3.35, y radius= 3.35]   ;
%Shape: Circle [id:dp9781129360661645] 
\draw   (569.77,78.98) .. controls (569.39,108.35) and (545.27,131.86) .. (515.89,131.48) .. controls (486.51,131.1) and (463.01,106.98) .. (463.39,77.6) .. controls (463.77,48.23) and (487.89,24.72) .. (517.27,25.1) .. controls (546.64,25.48) and (570.15,49.6) .. (569.77,78.98) -- cycle ;
%Shape: Circle [id:dp9555031747351268] 
\draw   (547.84,228.71) .. controls (547.59,247.59) and (532.08,262.7) .. (513.2,262.46) .. controls (494.32,262.21) and (479.2,246.71) .. (479.45,227.82) .. controls (479.69,208.94) and (495.2,193.83) .. (514.08,194.07) .. controls (532.97,194.32) and (548.08,209.82) .. (547.84,228.71) -- cycle ;
%Straight Lines [id:da19994917410370827] 
\draw    (479.45,227.82) -- (463.39,77.6) ;
%Straight Lines [id:da9713851507273499] 
\draw    (547.84,228.71) -- (569.77,78.98) ;
%Straight Lines [id:da1392924674648992] 
\draw    (532.48,256.51) ;
\draw [shift={(532.48,256.51)}, rotate = 0] [color={rgb, 255:red, 0; green, 0; blue, 0 }  ][fill={rgb, 255:red, 0; green, 0; blue, 0 }  ][line width=0.75]      (0, 0) circle [x radius= 3.35, y radius= 3.35]   ;
%Straight Lines [id:da552183072256714] 
\draw    (488.49,204.67) ;
\draw [shift={(488.49,204.67)}, rotate = 0] [color={rgb, 255:red, 0; green, 0; blue, 0 }  ][fill={rgb, 255:red, 0; green, 0; blue, 0 }  ][line width=0.75]      (0, 0) circle [x radius= 3.35, y radius= 3.35]   ;
%Straight Lines [id:da7651420147512551] 
\draw    (517.27,25.1) ;
\draw [shift={(517.27,25.1)}, rotate = 0] [color={rgb, 255:red, 0; green, 0; blue, 0 }  ][fill={rgb, 255:red, 0; green, 0; blue, 0 }  ][line width=0.75]      (0, 0) circle [x radius= 3.35, y radius= 3.35]   ;
%Straight Lines [id:da01132944802132485] 
\draw    (536.31,202.72) ;
\draw [shift={(536.31,202.72)}, rotate = 0] [color={rgb, 255:red, 0; green, 0; blue, 0 }  ][fill={rgb, 255:red, 0; green, 0; blue, 0 }  ][line width=0.75]      (0, 0) circle [x radius= 3.35, y radius= 3.35]   ;
%Straight Lines [id:da9662072911725705] 
\draw    (463.39,77.6) ;
\draw [shift={(463.39,77.6)}, rotate = 0] [color={rgb, 255:red, 0; green, 0; blue, 0 }  ][fill={rgb, 255:red, 0; green, 0; blue, 0 }  ][line width=0.75]      (0, 0) circle [x radius= 3.35, y radius= 3.35]   ;
%Straight Lines [id:da3780546884362026] 
\draw    (569.77,78.98) ;
\draw [shift={(569.77,78.98)}, rotate = 0] [color={rgb, 255:red, 0; green, 0; blue, 0 }  ][fill={rgb, 255:red, 0; green, 0; blue, 0 }  ][line width=0.75]      (0, 0) circle [x radius= 3.35, y radius= 3.35]   ;
%Straight Lines [id:da47955432880046067] 
\draw    (485.48,248.51) ;
\draw [shift={(485.48,248.51)}, rotate = 0] [color={rgb, 255:red, 0; green, 0; blue, 0 }  ][fill={rgb, 255:red, 0; green, 0; blue, 0 }  ][line width=0.75]      (0, 0) circle [x radius= 3.35, y radius= 3.35]   ;
%Curve Lines [id:da40268023545837917] 
\draw [color={rgb, 255:red, 208; green, 2; blue, 27 }  ,draw opacity=1 ]   (75.35,447.39) .. controls (76.35,422.39) and (82.35,375.39) .. (128.68,384.99) ;
%Curve Lines [id:da5657164489076383] 
\draw [color={rgb, 255:red, 208; green, 2; blue, 27 }  ,draw opacity=1 ]   (73.35,225.39) .. controls (85.35,143.39) and (77.35,26.39) .. (122.8,19) ;
%Curve Lines [id:da5437693248273088] 
\draw [color={rgb, 255:red, 208; green, 2; blue, 27 }  ,draw opacity=1 ]   (485.48,248.51) .. controls (467.35,83.39) and (482.35,36.39) .. (517.27,25.1) ;
%Curve Lines [id:da9732557248622546] 
\draw [color={rgb, 255:red, 208; green, 2; blue, 27 }  ,draw opacity=1 ]   (493.35,424.7) .. controls (474.35,430.39) and (473.35,365.39) .. (514.68,346) ;
%Straight Lines [id:da8757139942785441] 
\draw    (124.35,337.39) -- (124.35,288.39) ;
\draw [shift={(124.35,286.39)}, rotate = 90] [color={rgb, 255:red, 0; green, 0; blue, 0 }  ][line width=0.75]    (10.93,-3.29) .. controls (6.95,-1.4) and (3.31,-0.3) .. (0,0) .. controls (3.31,0.3) and (6.95,1.4) .. (10.93,3.29)   ;
\draw [shift={(124.35,339.39)}, rotate = 270] [color={rgb, 255:red, 0; green, 0; blue, 0 }  ][line width=0.75]    (10.93,-3.29) .. controls (6.95,-1.4) and (3.31,-0.3) .. (0,0) .. controls (3.31,0.3) and (6.95,1.4) .. (10.93,3.29)   ;
%Straight Lines [id:da9850788693435137] 
\draw    (212,151.99) -- (406.35,151.4) ;
\draw [shift={(408.35,151.39)}, rotate = 179.82] [color={rgb, 255:red, 0; green, 0; blue, 0 }  ][line width=0.75]    (10.93,-3.29) .. controls (6.95,-1.4) and (3.31,-0.3) .. (0,0) .. controls (3.31,0.3) and (6.95,1.4) .. (10.93,3.29)   ;
\draw [shift={(210,152)}, rotate = 359.82] [color={rgb, 255:red, 0; green, 0; blue, 0 }  ][line width=0.75]    (10.93,-3.29) .. controls (6.95,-1.4) and (3.31,-0.3) .. (0,0) .. controls (3.31,0.3) and (6.95,1.4) .. (10.93,3.29)   ;
%Straight Lines [id:da3190090416878264] 
\draw    (528.35,329.39) -- (528.35,280.39) ;
\draw [shift={(528.35,278.39)}, rotate = 90] [color={rgb, 255:red, 0; green, 0; blue, 0 }  ][line width=0.75]    (10.93,-3.29) .. controls (6.95,-1.4) and (3.31,-0.3) .. (0,0) .. controls (3.31,0.3) and (6.95,1.4) .. (10.93,3.29)   ;
\draw [shift={(528.35,331.39)}, rotate = 270] [color={rgb, 255:red, 0; green, 0; blue, 0 }  ][line width=0.75]    (10.93,-3.29) .. controls (6.95,-1.4) and (3.31,-0.3) .. (0,0) .. controls (3.31,0.3) and (6.95,1.4) .. (10.93,3.29)   ;
%Straight Lines [id:da22557955824113995] 
\draw    (221,409.99) -- (306.35,409.73) -- (415.35,409.4) ;
\draw [shift={(417.35,409.39)}, rotate = 179.82] [color={rgb, 255:red, 0; green, 0; blue, 0 }  ][line width=0.75]    (10.93,-3.29) .. controls (6.95,-1.4) and (3.31,-0.3) .. (0,0) .. controls (3.31,0.3) and (6.95,1.4) .. (10.93,3.29)   ;
\draw [shift={(219,410)}, rotate = 359.82] [color={rgb, 255:red, 0; green, 0; blue, 0 }  ][line width=0.75]    (10.93,-3.29) .. controls (6.95,-1.4) and (3.31,-0.3) .. (0,0) .. controls (3.31,0.3) and (6.95,1.4) .. (10.93,3.29)   ;

% Text Node
\draw (116,430.4) node [anchor=north west][inner sep=0.75pt]    {$\dotsc $};
% Text Node
\draw (119,387.4) node [anchor=north west][inner sep=0.75pt]    {$i_{1}$};
% Text Node
\draw (136,403.4) node [anchor=north west][inner sep=0.75pt]    {$i_{2}$};
% Text Node
\draw (110,405.4) node [anchor=north west][inner sep=0.75pt]    {$i_{p}$};
% Text Node
\draw (105,479.4) node [anchor=north west][inner sep=0.75pt]    {$\dotsc $};
% Text Node
\draw (178,349.4) node [anchor=north west][inner sep=0.75pt]    {$j_{1}$};
% Text Node
\draw (169.35,465.79) node [anchor=north west][inner sep=0.75pt]    {$j_{2}$};
% Text Node
\draw (52.35,450.79) node [anchor=north west][inner sep=0.75pt]    {$j_{q-1}$};
% Text Node
\draw (64.35,346.79) node [anchor=north west][inner sep=0.75pt]    {$j_{q}$};
% Text Node
\draw (506,427.4) node [anchor=north west][inner sep=0.75pt]    {$\dotsc $};
% Text Node
\draw (508,320.4) node [anchor=north west][inner sep=0.75pt]    {$i_{1}$};
% Text Node
\draw (518,408.4) node [anchor=north west][inner sep=0.75pt]    {$j_{2}$};
% Text Node
\draw (495.35,392.79) node [anchor=north west][inner sep=0.75pt]    {$j_{q}$};
% Text Node
\draw (502.68,478.79) node [anchor=north west][inner sep=0.75pt]    {$\dotsc $};
% Text Node
\draw (516.68,387.39) node [anchor=north west][inner sep=0.75pt]    {$j_{1}$};
% Text Node
\draw (586.35,400.79) node [anchor=north west][inner sep=0.75pt]    {$i_{2}$};
% Text Node
\draw (430.35,402.79) node [anchor=north west][inner sep=0.75pt]    {$i_{p}$};
% Text Node
\draw (93,256.88) node [anchor=north west][inner sep=0.75pt]    {$\dotsc $};
% Text Node
\draw (108,74.88) node [anchor=north west][inner sep=0.75pt]    {$\dotsc $};
% Text Node
\draw (529.02,127.88) node [anchor=north west][inner sep=0.75pt]  [rotate=-180.74]  {$\dotsc $};
% Text Node
\draw (518.62,268.78) node [anchor=north west][inner sep=0.75pt]  [rotate=-180.74]  {$\dotsc $};

\end{tikzpicture}
}
\end{center}
\end{rem}

We say that two arcs $a(i_1,j_1)[\lambda_1]$ and $a(i_2,j_2)[\lambda_2]$ \textbf{intersect nontrivially} if any two curves $\gamma_1 \in a(i_{1},j_{1})[\lambda_1]$ and $\gamma_2 \in a(i_2,j_2)[\lambda_2]$ intersect in their interiors. Otherwise we say that $a(i_1,j_1)[\lambda_1]$ and $a(i_2,j_2)[\lambda_2]$ \textbf{do not intersect nontrivially}. We say an arc $a(i,j)[\lambda]$ \textbf{self-intersects} if any two curves $\gamma_1 \in a(i,j)[\lambda]$ and $\gamma_2 \in a(i,j)[\lambda]$ intersect in their interiors. Otherwise, we say $a(i,j)[\lambda]$ \textbf{does not self-intersect}. If $a(i_1,j_1)[\lambda_1]$ and $a(i_2,j_2)[\lambda_2]$ do not intersect nontrivially, we say $a(i_1,j_1)[\lambda_1]$ is \textbf{clockwise} from $a(i_2,j_2)[\lambda_2]$ (or equivalently $a(i_2,j_2)[\lambda_2]$ is \textbf{counterclockwise} from $a(i_1,j_1)[\lambda_1]$) if and only if there exists $\gamma_1 \in a(i_{1},j_{1})[\lambda_1]$ and $\gamma_2 \in a(i_2,j_2)[\lambda_2]$ that do not intersect in their interiors and such that for all shared endpoints $p$, if we place a circle of radius ${r\over 2}$ about the shared point $p$, then the circle must be traversed clockwise through the interior of the annulus to get from the point of intersection of $a(i_2,j_2)[\lambda_2]$ with the circle to that of $a(i_1,j_1)[\lambda_1]$ with the circle. Again, we say that a collection of arcs $\{a(i_1,j_1)[\lambda_1], a(i_2,j_2)[\lambda_2], \dots , a(i_k,j_k)[\lambda_k]\}$ form a \textbf{cycle} if and only if $a(i_l,j_l)[\lambda_l]$ is clockwise from $a(i_{l+1},j_{l+1})[\lambda_{l+1}]$ for all $l < k$ and $a(i_k,j_k)[\lambda_k]$ is clockwise from $a(i_1,j_1)[\lambda_1]$. We say $a(i,j)[\lambda]$ forms a \textbf{loop} if it does not self-intersect and $i = j$.

\begin{defn}
An \textbf{exceptional arc diagram} (or \textbf{fundamental arc diagram}) is a collection $n$ of arcs on $A_{Q^{\bm{\varepsilon}}}$ that satisfies the following: 
\begin{enumerate}
\item no arc self-intersects or forms a loop,
\item distinct arcs do not intersect nontrivially, and 
\item the arcs do not form any cycles. 
\end{enumerate}

\end{defn}

Recall that the universal cover of the annulus is $\pi_{Q^{\bm\varepsilon}}:\mathbb{R} \times [-1,1]\rightarrow A_{Q^{\bm\varepsilon}}$. Notice that chord diagrams lie in the universal cover of the annulus. In particular, this means that any fundamental arc diagram on an annulus $A_{Q^{\bm{\varepsilon}}}$ lifts to a chord diagram $\pi^{-1}_{Q^{\bm{\varepsilon}}}(A_{Q^{\bm{\varepsilon}}})$ on $S_{n,+} \cup S_{n,-}$. Since chords cross in the interior of a fundamental domain of the universal cover, we have that two arcs on the annulus cross if and only if there exist some lifts such that the chords cross in the universal cover. Moreover, the clockwise nature of the arcs is preserved in the universal cover since the covering map is a local homeomorphism by definition. Therefore, arcs cross if and only if their corresponding lifted chords cross, and one arc is clockwise from the other if and only if the same is true for their lifts. From this we conclude that exceptional arc diagrams lift to fundamental chord diagrams and fundamental chord diagrams project to exceptional arc diagrams. Moreover, the above association of arcs to strands gives a bijection between arcs and fundamental lifts of string modules. We can explicitly write this bijection; however, we will introduce new notation that will be used in what follows.

\begin{defn}
For $l \in \mathbb{N}$ and $i,j\in[n]$, define the string module $(i,j;l)$ to be the module associated to the walk $e_{i+1}(\alpha_{i+1} \dots \alpha_i)^l\alpha_{i+1}\dots\alpha_{j-1}e_j$. 
\end{defn}

Notice that all strings in $Q$ can be uniquely written in this form. We take the convention that when we write strings in this notation, we take $l$ to be maximal. 

\begin{rem}
Using this notation, we have the following bijection $\varphi$ between strings and arcs:

 \[ M = (i,j;l) \mapsto \begin{cases} 
          a(i \, (\text{mod } n),j\, (\text{mod } n))[-l] & \text{if $M$ preinjective} \\
          a(i\, (\text{mod } n),j\, (\text{mod } n))[l] & \text{otherwise}
       \end{cases}.
    \]

We will denote by $a(i',j')[l]$ the arc associated to the module $(i,j;l)$ under this bijection. Notice that under this bijection, bridging arcs that begin on the inner (outer) boundary component correspond to preinjective (preprojective) modules. On the other hand, exterior arcs beginning and ending on the inner (outer) boundary component correspond to right (left) regular modules. 
\end{rem}

In some sense, these arc diagrams provide a better way to view exceptional collections since two arcs cross if and only if they cross when they are drawn.  Moreover, they give a nice geometric description of parametrized families of exceptional collections as we will see in Section \ref{sec: families of exceptional collections}. 

\begin{lem}\label{lem: regulars are small}
Let $M = (i,j;l)$ be a string module. Then $M$ is exceptional if and only if $a(i',j')[l]$ does not self-intersect and does not form a loop. 
\end{lem}

\begin{proof}
We begin by noticing that the arc $a(i',j')[l]$ self intersects or forms a loop if and only if it is an exterior arc (ie. $M$ is a regular module). 

We begin with the forward direction. If $M$ is preprojective or preinjective, then $a(i',j')[l]$ is a bridging arc, hence does not self intersect or form a loop. Whence suppose $M$ is regular. In this case, we will prove the contrapositive, so assume that $a(i',j')[l]$ either self intersects or forms a loop. If it self-intersects, we have that $l>0$. If $M$ is left-regular, so that $\tilde{\varepsilon}_{\tilde{i}} = + = \tilde{\varepsilon}_{\tilde{j}}$, we can write $M = e_{i+1}(\alpha_{i+1} \dots \alpha_i)^l\alpha_{i+1}\dots\alpha_{j-1}e_j$. Since $l \geq 1$, we have the following two factorizations of $M$: $F = (e_{i+1},\alpha_{i+1}\dots\alpha_{j-1},\alpha_{j}\dots \alpha_{j-1}e_j)$ and $S = ((\alpha_{i+1}\dots\alpha_{i})^l,\alpha_{i+1}\dots\alpha_{j-1},e_j)$. Since $M$ is left regular, we have that $\alpha_i,\alpha_j\in Q_1$. This allows us to conclude that $F$ is a quotient factorization while $S$ is a submodule factorization. Moreover, we have that this admissible pair is a two-sided graph map by definition. Thus by Theorem \ref{thm: basis for ext}, $\text{Ext}(M,M)\neq 0$ and we have that $M$ is not exceptional. Now suppose that $a(i',j')[l]$ does not self-intersect, but forms a loop. Thus $i = j$ and we can write $M = e_{i+1}\alpha_{i+1} \dots \alpha_{i-1}e_i$. We see that $M$ is connectable to itself; that is, $M\alpha_iM$ is a string in $Q$. Therefore by Theorem \ref{thm: basis for ext}, $\text{Ext}(M,M)\neq 0$ and $M$ is not exceptional. The proof when $M$ is right regular follows from duality.

Conversely, suppose that $a(i',j')[l]$ neither self-intersects nor forms a loop. If $M$ is not regular, it is well known that $M$ is exceptional. Whence suppose $M$ is regular lying in a tube of rank $r$. By assumption, we have that $i\neq j$ and $l = 0$. It follows that $rl(M) < r$, so by Lemma \ref{lem: hom and ext for regulars}, $M$ is exceptional. 
\end{proof}

We have the following reformulations of Lemma \ref{lem: key lemma} and Theorem \ref{thm: bijection with strand diagrams}: 

\begin{lem} \label{lem: key lemma on annulus}
Let $Q^{\bm{\varepsilon}}$ be a quiver of type $\tilde{\mathbb{A}}_{n-1}$ and let $U,V \in \text{ind}(\text{rep}(Q^{\bm{\varepsilon}}))$ be two exceptional string modules. 
\begin{enumerate}
\item The arcs corresponding to $U$ and $V$ intersect nontrivially or form a cycle if and only if neither $(U,V)$ nor $(V,U)$ are exceptional pairs. 
\item The arc corresponding to $U$ is clockwise from that corresponding to $V$if and only if $(U,V)$ is an exceptional pair and $(V,U)$ is not. 
\item The arcs corresponding to $U$ and $V$ do not intersect at any of their endpoints and they do not intersect nontrivially if and only if both $(U,V)$ and $(V,U)$ form exceptional pairs.  
\end{enumerate}
\end{lem}

\begin{proof}
Let $U$ and $V$ be two exceptional string modules with associated arcs $a(i,j)[k]$ and $a(r,s)[l]$ respectively. Fix $X$ to be to be the fundamental strand on $S_{n,\tilde{\varepsilon}}$ corresponding to the fundamental lift of $U$. Let $Y_m$ be a strand on $S_{n,\tilde{\varepsilon}}$ associated to some string module in $F^{-1}(V)$ for all $m$, where $F:\tilde{Q}^{\tilde{\bm\varepsilon}}\rightarrow Q^{\bm\varepsilon}$ is the covering functor. We first prove the forward direction.
\begin{enumerate}
\item Suppose first that the arcs $a(i,j)[k]$ and $a(r,s)[l]$ intersect nontrivially. Then there exists an $m$ such that $X$ and $Y_m$ intersect non-trivially in the universal cover. Then by Lemma 11 in [\ref{ref: combinatorics exceptional sequences type A}] either Ext$_{\Bbbk\tilde{Q}}(X,Y_i) \neq 0 \neq$ Hom$_{\Bbbk\tilde{Q}}(Y_i,X)$ or Ext$_{\Bbbk\tilde{Q}}(Y_i,X) \neq 0 \neq$ Hom$_{\Bbbk\tilde{Q}}(X,Y_i)$. In either case by definition of the covering functor $F$ along with the AR formulas as noted in Remark \ref{rem: counting dimensions of hom and ext in univ cover}, we have that neither $(U,V)$ nor $(V,U)$ form an exceptional $\Bbbk Q$ pair. Now suppose the arcs form a cycle. Then there exist lifts $Y_m$ and $Y_p$ such that $X$ is locally clockwise from $Y_m$ and locally counter clockwise from $Y_p$. By Lemma 11 in [\ref{ref: combinatorics exceptional sequences type A}], the fact that $X$ is locally clockwise from $Y_m$ implies that Hom$_{\Bbbk\tilde{Q}}(Y_m,X) \neq 0$ or Ext$_{\Bbbk\tilde{Q}}(Y_m,X) \neq 0$. Similarly, we have that Hom$_{\Bbbk\tilde{Q}}(X,Y_p) \neq 0$ or Ext$_{\Bbbk\tilde{Q}}(X,Y_p) \neq 0$. By Remark \ref{rem: counting dimensions of hom and ext in univ cover}, we have that either Hom$_{\Bbbk Q}(V,U) \neq 0$ or Ext$_{\Bbbk Q}(V,U) \neq 0$ and either Hom$_{\Bbbk Q}(U,V) \neq 0$ or Ext$_{\Bbbk Q}(U,V) \neq 0$. We conclude that neither $(U,V)$ nor $(V,U)$ form an exceptional sequence. 

\item Suppose now that $a(i,j)[k]$ is clockwise from $a(r,s)[l]$. If these arcs share one endpoint, there is exactly one lift $Y_m$ such that $X$ is clockwise of $Y_m$ and all other lifts do not interact with $X$. Thus by Lemma 11 in [\ref{ref: combinatorics exceptional sequences type A}], $(U,V)$ forms an exceptional pair and $(V,U)$ does not. If these arcs share two endpoints, then there exist precisely two lifts $Y_m$ and $Y_p$ such that $X$ is locally clockwise from $Y_m$ and locally clockwise from $Y_p$. By Lemma 11 in [\ref{ref: combinatorics exceptional sequences type A}], in either case, Hom$_{\Bbbk\tilde{Q}}(Y_{m(p)},X) \neq 0$ or $0 \neq$ Ext$_{\Bbbk\tilde{Q}}(Y_{m(p)},X)$ and Hom$_{\Bbbk\tilde{Q}}(X,Y_{m(p)}) = 0 =$ Ext$_{\Bbbk\tilde{Q}}(X,Y_{m(p)})$. Thus $(U,V)$ is exceptional and $(V,U)$ is not.

\item Suppose the arcs neither intersect nor trivially intersect. Then for all $Y_m \in F^{-1}(V)$, the strands $X$ and $Y_m$ neither intersect, nor trivially intersect. Thus by Lemma 11 in [\ref{ref: combinatorics exceptional sequences type A}], for all $m$, Hom$_{\Bbbk\tilde{Q}}(X,Y_m) =$ Ext$_{\Bbbk\tilde{Q}}(X,Y_m) =$ Hom$_{\Bbbk\tilde{Q}}(Y_m,X) =$ Ext$_{\Bbbk\tilde{Q}}(X,Y_m) = 0$. We conclude that both $(U,V)$ and $(V,U)$ form exceptional pairs.
\end{enumerate}
We prove the reverse direction by contrapositive.
\begin{enumerate}
\item Suppose the arcs $a(i,j)[k]$ and $a(r,s)[l]$ do not intersect nontrivially and do not form a cycle. Then we have the following cases.

\begin{enumerate}
\item Suppose that $a(i,j)[k]$ is clockwise $a(r,s)[l]$. Then by part 2 of the forward direction, $(U,V)$ is exceptional and $(V,U)$ is not. Similarly, if the arc $a(r,s)[l]$ is clockwise of $a(r,s)[l]$, then $(V,U)$ is exceptional and $(U,V)$ is not.
\item Suppose the arcs $a(i,j)[k]$ and $a(r,s)[l]$ neither intersect nor intersect trivially. Then by part 3 of the forward direction, both $(U,V)$ and $(V,U)$ are exceptional.
\end{enumerate}
In either case, we have shown that either $(U,V)$ or $(V,U)$, or both are exceptional.

\item Suppose the arc $a(i,j)[k]$ is not clockwise from $a(r,s)[l]$. We have the following cases:

\begin{enumerate}
\item Suppose $a(i,j)[k]$ is counterclockwise from $a(r,s)[l]$. Then by part 2 of the forward direction, $(V,U)$ is an exceptional pair and $(U,V)$ is not.
\item Suppose the arcs cross. Then by part 1 of the forward direction, neither $(U,V)$ nor $(V,U)$ form an exceptional pair.
\item Suppose the arcs neither intersect nor trivially intersect. Then by part 3 of the forward direction, both $(U,V)$ and $(V,U)$ form exceptional pairs.
\end{enumerate}
In all cases, we have shown that either $(U,V)$ is not exceptional and $(V,U)$ is exceptional, $(U,V)$ and $(V,U)$ are both exceptional, or $(U,V)$ and $(V,U)$ are both not exceptional.

\item Suppose the arcs either intersect or trivially intersect. If they intersect or form a cycle, then by part 1 of the forward direction, neither $(U,V)$ nor $(V,U)$ form an exceptional pair. If $a(i,j)[k]$ is clockwise from $a(r,s)[l]$, then by part 2 of the forward direction, $(U,V)$ is exceptional and $(V,U)$ is not. Similarly, if $a(r,s)[l]$ is clockwise from $a(i,j)[k]$, then by part 2 of the forward direction, $(V,U)$ is exceptional and $(U,V)$ is not. In any case, either $(U,V)$, or $(V,U)$, or both pairs are not exceptional.
\end{enumerate}
\end{proof}

The proof of the next theorem is very similar to that of Theorem 12 in [\ref{ref: combinatorics exceptional sequences type A}].

\begin{thm}\label{thm: ec bijection with arc diagrams}
Let $\bar{E}_{\tilde{\varepsilon}} := \{ \text{exceptional collections of } Q^{\bm{\varepsilon}}\}$ where $Q^{\bm{\varepsilon}}$ is a quiver of type $\tilde{\mathbb{A}}_{n-1}$. There is a bijection between the set of complete exceptional collections of $\Bbbk Q^{\bm{\varepsilon}}$-modules and fundamental arc diagrams on $A_{Q^{\bm{\varepsilon}}}$.
\end{thm}

\begin{proof}
Let $\xi = \{(i_1,j_1;l_1),\dots (i_n,j_n;l_n)\}$ be an exceptional collection of $Q^{\bm{\varepsilon}}$ modules and consider the arc diagram $\varphi(\xi) = D$ consisting of the arcs $\{a(i'_1,j'_1)[l_1],\dots a(i'_n,j'_n)[l_n]\}$ on the annulus $A_{Q^{\bm{\varepsilon}}}$. Since $\xi$ is exceptional, by Lemma \ref{lem: regulars are small}, we have that no arc self-intersects or forms a loop. Moreover, by 1. of Lemma \ref{lem: key lemma on annulus}, we see that no two arcs cross or form a cycle. We will now show that $D$ does not contain a cycle by contradiction. Suppose the arcs $a(i'_s,j'_s)[l_s],\dots a(i'_p,j'_p)[l_p]$ form a cycle. Then we have that for all $x\in[p]$, the arc $a(i'_x,j'_x)[\lambda_x]$ is clockwise from the arc $a(i'_{x+1},j'_{x+1})[\lambda_{x+1}]$. By 2. of Lemma \ref{lem: key lemma on annulus}, this implies that for all $x\in[p]$, $(i_x,j_x;l_x)$ proceeds $(i_{x+1},j_{x+1};l_{x+1})$ in any exceptional sequence corresponding to $\xi$. But this is impossible, contradicting the exceptionality of $\xi$.

We will now show that any exceptional arc diagrams yields an exceptional sequence, so suppose $D$ is an exceptional arc diagram. Since no arc self-intersects or forms a loop and moreover, no two arcs cross or form a cycle, it follows that either $(\varphi^{-1}(a(i'_s,j'_s)[\lambda_s]),\varphi^{-1}(a(i'_x,j'_x)[\lambda_x]))$ or $(\varphi^{-1}(a(i'_x,j'_x)[\lambda_x]),\varphi^{-1}(a(i'_s,j'_s)[\lambda_s]))$ forms an exceptional pair for all $s\neq x$. Now, there exists an arc $a(i'_{s_1},j'_{s_1})[\lambda_{s_1}] \in D$ such that $(\varphi^{-1}(a(i'_{s_1},j'_{s_1})[\lambda_{s_1}]),\varphi^{-1}(a(i'_s,j'_s)[\lambda_s]))$ is an exceptional pair for all $a(i'_s,j'_s)[\lambda_s]\in D - \{a(i'_{s_1},j'_{s_1})[\lambda_{s_1}]\}$; for if not, $D$ would contain a cycle. By continuing to choose arcs in this way inductively, we obtain an exceptional sequence as desired.
\end{proof}

At this point, we will provide some intuition behind the results of Lemma \ref{lem: key lemma on annulus}.

\begin{prop} Let $Q^{\bm{\varepsilon}}$ be a quiver of type $\tilde{\mathbb{A}}_{n-1}$, let $U,V \in \text{ind}(\text{rep}(Q^{\bm{\varepsilon}}))$ be two exceptional string modules, and let $a_1$ and $a_2$ be their corresponding arcs on $A_{Q^{\bm\varepsilon}}$.
\begin{enumerate}

\item The arcs $a_1$ and $a_2$ do not intersect at any of their endpoints and they do not intersect nontrivially if and only if there are no extensions or morphisms between $U$ and $V$.

\item The arc $a_1$ is clockwise from $a_2$ if and only if Hom$(U,V)$ contains only non two-sided graph maps and Ext$(U,V) = \text{Hom}(V,U) = \text{Ext}(V,U) = 0$, or $U$ and $V$ are connectable and Hom$(U,V) =\text{Hom}(V,U) = \text{Ext}(V,U) = 0$.

\item The arcs $a_1$ and $a_2$ form a cycle if and only if $U$ is connectable to $V$, $V$ is connectable to $U$, and Hom$(U,V) = 0 = \text{Hom}(V,U)$. 

\item The arcs $a_1$ and $a_2$ intersect nontrivially if and only if there is a two-sided graph map between $U$ and $V$.

\end{enumerate}
\end{prop}

\begin{proof}
\begin{enumerate}
\item This follows immediately from Lemma \ref{lem: key lemma on annulus}.

\item In this case, there are 6 possible local configurations for $a_1$ and $a_2$, as seen in Figure \ref{fig: 6 possible local configurations on Annulus}. We begin with the annulus on the top left of Figure \ref{fig: 6 possible local configurations on Annulus}. Suppose the arcs locally lie in this configuration, and suppose $t(U) = i = s(V)$. Then we have that $U\alpha_iV$ is a string in $Q$, hence $U$ is connectable to $V$. Since the arcs don't cross or form a cycle, we have that $(U,V)$ is exceptional and $(V,U)$ is not by Lemma \ref{lem: key lemma on annulus}. Thus Hom$(V,U) = 0 = \text{Ext}(V,U)$. If $U$ is preinjective, then $V$ must be regular or preprojective for the arcs to attain this configuration, in which case Hom$(U,V) = 0$. If $U$ is regular, then $V$ is either regular or preprojective. If $V$ is preprojective, then Hom$(U,V) = 0$. If $V$ is regular, then since the arcs don't cross or form a cycle, the modules don't share support. Thus Hom$(U,V) = 0$ in all cases. The case in which the arcs appear as in the bottom left annulus is analogous. 

Conversely, suppose $U \alpha V$ is a string in $Q$, there is no Hom between $U$ and $V$, and Ext$(V,U)=0$. Then $t(U) = s(V)$, so the corresponding arcs must locally appear in one of the two configurations on the left hand side of Figure \ref{fig: 6 possible local configurations on Annulus}. Moreover, we have that $(U,V)$ is exceptional and $(V,U)$ is not. By Lemma \ref{lem: key lemma on annulus}, we conclude that the arc $a_1$ is clockwise of $a_2$.\\

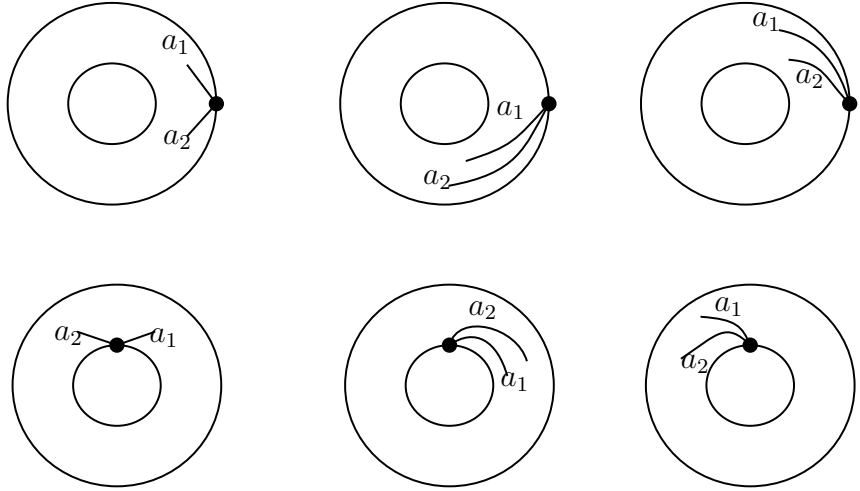
\begin{figure}
\begin{center}

\tikzset{every picture/.style={line width=0.75pt}} %set default line width to 0.75pt        

\begin{tikzpicture}[x=0.75pt,y=0.75pt,yscale=-1,xscale=1]
%uncomment if require: \path (0,370); %set diagram left start at 0, and has height of 370

%Shape: Donut [id:dp6630301983698756] 
\draw   (90.54,86.29) .. controls (90.54,75.05) and (100.43,65.93) .. (112.63,65.93) .. controls (124.83,65.93) and (134.72,75.05) .. (134.72,86.29) .. controls (134.72,97.54) and (124.83,106.65) .. (112.63,106.65) .. controls (100.43,106.65) and (90.54,97.54) .. (90.54,86.29)(60,86.29) .. controls (60,58.18) and (83.56,35.39) .. (112.63,35.39) .. controls (141.69,35.39) and (165.26,58.18) .. (165.26,86.29) .. controls (165.26,114.4) and (141.69,137.19) .. (112.63,137.19) .. controls (83.56,137.19) and (60,114.4) .. (60,86.29) ;
%Straight Lines [id:da4449928164505501] 
\draw    (165.26,86.29) ;
\draw [shift={(165.26,86.29)}, rotate = 0] [color={rgb, 255:red, 0; green, 0; blue, 0 }  ][fill={rgb, 255:red, 0; green, 0; blue, 0 }  ][line width=0.75]      (0, 0) circle [x radius= 3.35, y radius= 3.35]   ;
%Shape: Donut [id:dp28082307139359597] 
\draw   (93.01,228.49) .. controls (93.01,217.25) and (102.9,208.13) .. (115.1,208.13) .. controls (127.3,208.13) and (137.19,217.25) .. (137.19,228.49) .. controls (137.19,239.74) and (127.3,248.85) .. (115.1,248.85) .. controls (102.9,248.85) and (93.01,239.74) .. (93.01,228.49)(62.47,228.49) .. controls (62.47,200.38) and (86.03,177.6) .. (115.1,177.6) .. controls (144.16,177.6) and (167.72,200.38) .. (167.72,228.49) .. controls (167.72,256.6) and (144.16,279.39) .. (115.1,279.39) .. controls (86.03,279.39) and (62.47,256.6) .. (62.47,228.49) ;
%Straight Lines [id:da8789871869130903] 
\draw    (115.1,208.13) ;
\draw [shift={(115.1,208.13)}, rotate = 0] [color={rgb, 255:red, 0; green, 0; blue, 0 }  ][fill={rgb, 255:red, 0; green, 0; blue, 0 }  ][line width=0.75]      (0, 0) circle [x radius= 3.35, y radius= 3.35]   ;
%Straight Lines [id:da08616256492089236] 
\draw    (150.46,66.58) -- (165.26,86.29) ;
%Straight Lines [id:da8247080042058517] 
\draw    (165.26,86.29) -- (150.46,102.75) ;
%Straight Lines [id:da08880302932542095] 
\draw    (94.96,201.3) -- (115.1,208.13) ;
%Straight Lines [id:da5932493321762904] 
\draw    (115.1,208.13) -- (134.43,201.3) ;
%Shape: Donut [id:dp8142778592200297] 
\draw   (258.26,86.29) .. controls (258.26,75.05) and (268.15,65.93) .. (280.35,65.93) .. controls (292.55,65.93) and (302.44,75.05) .. (302.44,86.29) .. controls (302.44,97.54) and (292.55,106.65) .. (280.35,106.65) .. controls (268.15,106.65) and (258.26,97.54) .. (258.26,86.29)(227.72,86.29) .. controls (227.72,58.18) and (251.28,35.39) .. (280.35,35.39) .. controls (309.42,35.39) and (332.98,58.18) .. (332.98,86.29) .. controls (332.98,114.4) and (309.42,137.19) .. (280.35,137.19) .. controls (251.28,137.19) and (227.72,114.4) .. (227.72,86.29) ;
%Straight Lines [id:da756957271049542] 
\draw    (332.98,86.29) ;
\draw [shift={(332.98,86.29)}, rotate = 0] [color={rgb, 255:red, 0; green, 0; blue, 0 }  ][fill={rgb, 255:red, 0; green, 0; blue, 0 }  ][line width=0.75]      (0, 0) circle [x radius= 3.35, y radius= 3.35]   ;
%Shape: Donut [id:dp18731012859091356] 
\draw   (260.73,228.49) .. controls (260.73,217.25) and (270.62,208.13) .. (282.82,208.13) .. controls (295.02,208.13) and (304.91,217.25) .. (304.91,228.49) .. controls (304.91,239.74) and (295.02,248.85) .. (282.82,248.85) .. controls (270.62,248.85) and (260.73,239.74) .. (260.73,228.49)(230.19,228.49) .. controls (230.19,200.38) and (253.75,177.6) .. (282.82,177.6) .. controls (311.88,177.6) and (335.44,200.38) .. (335.44,228.49) .. controls (335.44,256.6) and (311.88,279.39) .. (282.82,279.39) .. controls (253.75,279.39) and (230.19,256.6) .. (230.19,228.49) ;
%Straight Lines [id:da0419040790551648] 
\draw    (282.82,208.13) ;
\draw [shift={(282.82,208.13)}, rotate = 0] [color={rgb, 255:red, 0; green, 0; blue, 0 }  ][fill={rgb, 255:red, 0; green, 0; blue, 0 }  ][line width=0.75]      (0, 0) circle [x radius= 3.35, y radius= 3.35]   ;
%Shape: Donut [id:dp027035376296954405] 
\draw   (409.95,86.29) .. controls (409.95,75.05) and (419.84,65.93) .. (432.04,65.93) .. controls (444.24,65.93) and (454.13,75.05) .. (454.13,86.29) .. controls (454.13,97.54) and (444.24,106.65) .. (432.04,106.65) .. controls (419.84,106.65) and (409.95,97.54) .. (409.95,86.29)(379.41,86.29) .. controls (379.41,58.18) and (402.97,35.39) .. (432.04,35.39) .. controls (461.1,35.39) and (484.67,58.18) .. (484.67,86.29) .. controls (484.67,114.4) and (461.1,137.19) .. (432.04,137.19) .. controls (402.97,137.19) and (379.41,114.4) .. (379.41,86.29) ;
%Straight Lines [id:da9926097696748768] 
\draw    (484.67,86.29) ;
\draw [shift={(484.67,86.29)}, rotate = 0] [color={rgb, 255:red, 0; green, 0; blue, 0 }  ][fill={rgb, 255:red, 0; green, 0; blue, 0 }  ][line width=0.75]      (0, 0) circle [x radius= 3.35, y radius= 3.35]   ;
%Shape: Donut [id:dp7548887998282148] 
\draw   (412.41,228.49) .. controls (412.41,217.25) and (422.3,208.13) .. (434.5,208.13) .. controls (446.7,208.13) and (456.59,217.25) .. (456.59,228.49) .. controls (456.59,239.74) and (446.7,248.85) .. (434.5,248.85) .. controls (422.3,248.85) and (412.41,239.74) .. (412.41,228.49)(381.88,228.49) .. controls (381.88,200.38) and (405.44,177.6) .. (434.5,177.6) .. controls (463.57,177.6) and (487.13,200.38) .. (487.13,228.49) .. controls (487.13,256.6) and (463.57,279.39) .. (434.5,279.39) .. controls (405.44,279.39) and (381.88,256.6) .. (381.88,228.49) ;
%Straight Lines [id:da28425532189273994] 
\draw    (434.5,208.13) ;
\draw [shift={(434.5,208.13)}, rotate = 0] [color={rgb, 255:red, 0; green, 0; blue, 0 }  ][fill={rgb, 255:red, 0; green, 0; blue, 0 }  ][line width=0.75]      (0, 0) circle [x radius= 3.35, y radius= 3.35]   ;
%Curve Lines [id:da11365454108205575] 
\draw    (282.42,127.7) .. controls (316.95,121.46) and (320.65,108.99) .. (332.98,86.29) ;
%Curve Lines [id:da05442582331536183] 
\draw    (291.05,115.23) .. controls (314.48,107.74) and (315.71,105.25) .. (332.98,86.29) ;
%Curve Lines [id:da5410522670512821] 
\draw    (282.82,208.13) .. controls (295.98,198.8) and (307.08,206.29) .. (312.01,223.75) ;
%Curve Lines [id:da6328289287423035] 
\draw    (282.82,208.13) .. controls (287.35,191.32) and (316.95,198.8) .. (321.88,216.27) ;
%Curve Lines [id:da3936005838262626] 
\draw    (448.9,49.11) .. controls (469.87,52.86) and (479.73,68.83) .. (484.67,86.29) ;
%Curve Lines [id:da45071975845951995] 
\draw    (453.84,64.08) .. controls (472.33,65.33) and (472.33,74.06) .. (484.67,86.29) ;
%Curve Lines [id:da3085772893134282] 
\draw    (409.44,193.81) .. controls (427.94,195.06) and (430.4,198.8) .. (434.5,208.13) ;
%Curve Lines [id:da13035591404631464] 
\draw    (399.57,215.02) .. controls (418.07,202.54) and (422.17,195.91) .. (434.5,208.13) ;

% Text Node
\draw (136,50.4) node [anchor=north west][inner sep=0.75pt]    {$a_{1}$};
% Text Node
\draw (137,98.4) node [anchor=north west][inner sep=0.75pt]    {$a_{2}$};
% Text Node
\draw (130,198.4) node [anchor=north west][inner sep=0.75pt]    {$a_{1}$};
% Text Node
\draw (305,83.4) node [anchor=north west][inner sep=0.75pt]    {$a_{1}$};
% Text Node
\draw (307,220.4) node [anchor=north west][inner sep=0.75pt]    {$a_{1}$};
% Text Node
\draw (434.04,38.79) node [anchor=north west][inner sep=0.75pt]    {$a_{1}$};
% Text Node
\draw (415,182.4) node [anchor=north west][inner sep=0.75pt]    {$a_{1}$};
% Text Node
\draw (82,197.4) node [anchor=north west][inner sep=0.75pt]    {$a_{2}$};
% Text Node
\draw (268,118.4) node [anchor=north west][inner sep=0.75pt]    {$a_{2}$};
% Text Node
\draw (291,184.4) node [anchor=north west][inner sep=0.75pt]    {$a_{2}$};
% Text Node
\draw (455.84,67.48) node [anchor=north west][inner sep=0.75pt]    {$a_{2}$};
% Text Node
\draw (398,211.4) node [anchor=north west][inner sep=0.75pt]    {$a_{2}$};

\end{tikzpicture}

\end{center}
\caption{The six possible local configurations of an arc $a_1$ being clockwise from another arc $a_2$ as in Figure \ref{fig: six local clockwise configs}.}
\label{fig: 6 possible local configurations on Annulus}
\end{figure}

Now suppose the arcs locally appear as the middle or right column of Figure \ref{fig: 6 possible local configurations on Annulus}. Then by Lemma \ref{lem: key lemma on annulus}, $(U,V)$ is exceptional and $(V,U)$ is not. Thus there is a nontrivial morphism from $U$ to $V$, or a nontrivial extension of $U$ by $V$, or both. Since $(U,V)$ is exceptional, Hom$(V,U) = 0 = \text{Ext}(V,U)$. From this we conclude that there is not a two-sided graph map between $U$ and $V$. Thus if there is an extension of $U$ by $V$, we must have that $U$ is connectable to $V$. By the previous paragraph, this allows us to conclude that the arcs appear on the other boundary circle as in the left hand side of Figure \ref{fig: 6 possible local configurations on Annulus}, which is not possible since each arc has only one start and end point. Thus we conclude there is a graph map between $U$ and $V$ that is not two-sided. 

Conversely, suppose Hom$(U,V)$ contains only non two-sided graph maps and Ext$(U,V) = \text{Hom}(V,U) = \text{Ext}(V,U) = 0$. Thus $(U,V)$ is exceptional and $(V,U)$ is not. By Lemma \ref{lem: key lemma on annulus}, the arc $U$ must be clockwise of $V$, and by the first paragraph, they must lie as in the middle or right column of Figure \ref{fig: 6 possible local configurations on Annulus}. \\

\item Suppose the arcs $U$ and $V$ form a cycle. Then they must both be exterior arcs on the same boundary component and share two endpoints. Moreover, at each shared endpoint, the arcs must appear as in the left hand side of Figure \ref{fig: 6 possible local configurations on Annulus}. From this we conclude that $U$ is connectable to $V$ and $V$ is connectable to $U$. Since the corresponding modules have disjoint support, it follows that there is no Hom between them.

Conversely, suppose that $U$ is connectable to $V$, $V$ is connectable to $U$, and Hom$(U,V) = 0 = \text{Hom}(V,U)$. Since $U$ is connectable to $V$ and $V$ is connectable to $U$, the arcs must locally appear at their shared endpoints at in the left hand side of Figure \ref{fig: 6 possible local configurations on Annulus}. Moreover, since there is no Hom, we conclude that the arcs don't cross, and thus form a cycle.

\item Suppose there is a two-sided graph map from $U$ to $V$. Then Hom$(U,V) \neq 0 \neq \text{Ext}(V,U)$. By Lemma \ref{lem: key lemma on annulus}, we have that either the arcs corresponding to $U$ and $V$ form a cycle, or cross. Since there is Hom between them, they can't form a cycle by part 3, so the arcs must cross.

Conversely, suppose that the arcs corresponding to $U$ and $V$ cross. Since the arcs cross, we may without loss of generality assume that Hom$(U,V) \neq \text{Ext}(V,U)$. If the extension comes from a two-sided graph map, we are done, so suppose that the extension does not come from a two-sided graph map. Then $V$ is connectable to $U$, and there is moreover a non two-sided graph map from $U$ to $V$. Thus on one boundary component, the arcs must appear locally as in the middle or right of Figure \ref{fig: 6 possible local configurations on Annulus} where on the other boundary component, they must appear as on the left of \ref{fig: 6 possible local configurations on Annulus}, which is impossible. We conclude that there must be a two-sided graph map between $U$ and $V$.

\end{enumerate}
\end{proof}

\section{Families of Exceptional Collections}\label{sec: families of exceptional collections}

\indent

It is well known that there are infinitely many exceptional collections of type $\tilde{\mathbb{A}}$. In this section we will classify them into families and show that there are finitely many such families. Recall that $Q^{\bm{\varepsilon}}$ is a quiver of type $\tilde{\mathbb{A}}_{n-1}$ and $A_{Q^{\bm\varepsilon}}$ is the associated annulus. 

\subsection{A Geometric Description of Parameterized Families}

\indent

In this subsection, we will define parametrized families of exceptional collections in terms of arc diagrams. We begin by defining what we mean by Dehn twists of the annulus, then we will use these twists to define the families.

A \textbf{$2\pi$ clockwise Dehn twist} of the inner circle of the annulus $A_{Q^{\bm\varepsilon}}$ is a homeomorphism between arc diagrams $D_1 \rightarrow D_2$ defined on arcs as follows:
 
\[ M = a(i,j)[l] \mapsto \begin{cases} 
          a(i,j)[l] & \text{if $a(i,j)[l]$ is an exterior arc (ie $M$ is a regular module)} \\
          a(j,i)[0] & \text{if $a(i,j)[l]$ is a bridging arc with $M$ injective and $l =0$} \\
          a(i,j)[l+1] & \text{otherwise}		  
       \end{cases}.
    \]

A \textbf{$2\pi$ counter clockwise Dehn twist} of the inner circle of the annulus is the inverse of a $2\pi$ clockwise Dehn twist. An example of these twists can be seen in Examples \ref{exmp: dehn twist big ec} and \ref{exam: tranjective component}.

\begin{defn}\label{defn: equivalent arc diagrams}
Let $\xi_1$ and $\xi_2$ be two exceptional collections and $D_1$ and $D_2$ be their corresponding exceptional arc diagrams. We say the exceptional collections are \textbf{equivalent}, denoted $\xi_1 \sim \xi_2$, if $D_2$ is attained from $D_1$ by a sequence of $2\pi$ Dehn twists of the inner circle of the annulus.  
\end{defn}

\begin{exmp}\label{exmp: dehn twist big ec}

We have on the left the corresponding exceptional arc diagram from Example \ref{exmp: fundamental chord diagram} and on the right, a clockwise $2\pi$ Dehn twist of the inner boundary circle corresponding to another exceptional arc diagram in same equivalence class.

\begin{center}

\includegraphics[height = 5cm, width = 12cm]{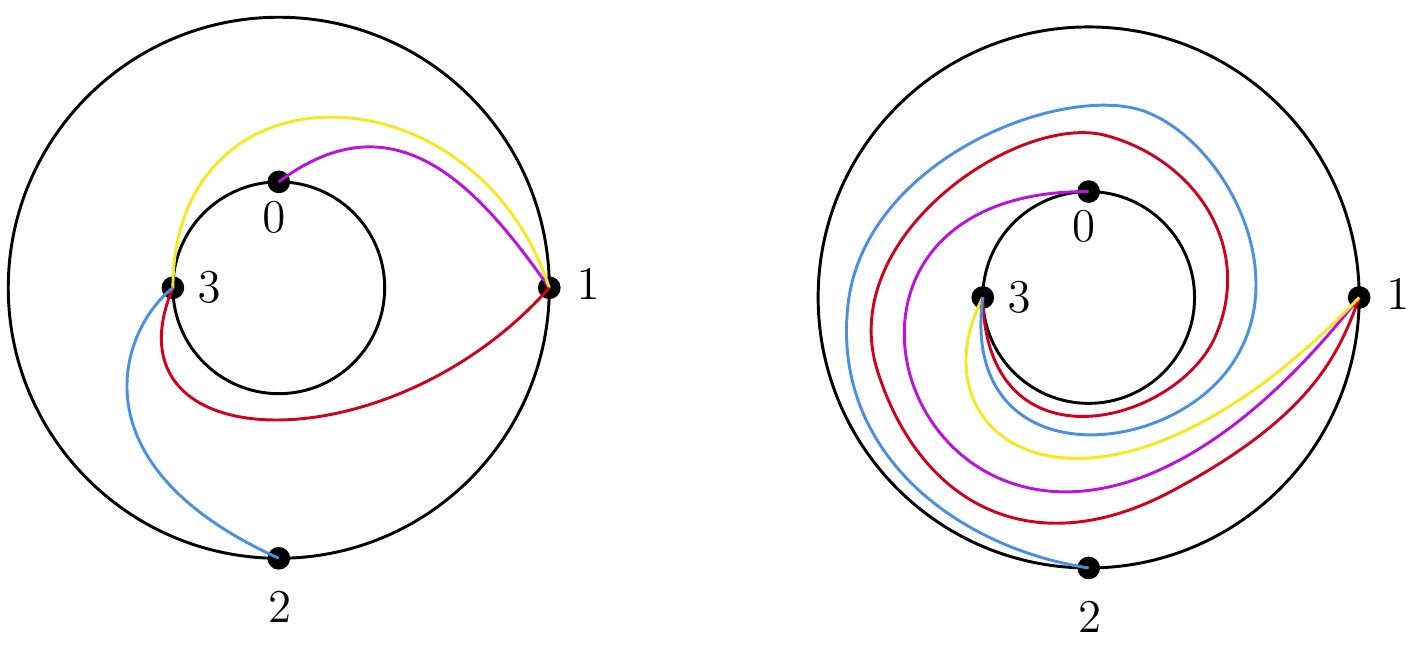}

\end{center}
\end{exmp}

Since $2\pi$ Dehn twists are homeomorphisms, note that they preserve exceptionality of the arc diagram, and hence the exceptional collection by Theorem \ref{thm: ec bijection with arc diagrams}. Before proving that there are only finitely many equivalence classes of exceptional collections under the equivalence relation in Definition \ref{defn: equivalent arc diagrams}, we will study the conditions that the winding numbers of the arcs must satisfy to ensure that the diagram is exceptional.

As a result of Lemma \ref{lem: regulars are small} along with the fact that $2\pi$ Dehn twists fix exterior arcs, we adopt the convention that exterior arcs (regular modules) will be written $a(i,j)$; that is, without indicating their winding number of zero. 

\begin{lem}\label{lem: only projectives and injectives in same fam}
Let $\xi = \{a(i_1,j_1)[\lambda_1], \dots, a(i_k,j_k)[\lambda_k]\}$ be a not necessarily complete exceptional collection without regular modules containing both preprojective and preinjective modules. Then $\lambda_p = 0$ for all $p$. 
\end{lem} 

\begin{proof}
We prove the contrapositive. Let $U = (i,j;k)$ be preprojective with $k\geq1$ and $V = (l,m;0)$ be injective. Then we can decompose $V = (e_{s(\alpha'_{l+1})},\alpha'_{l+1}\dots\alpha'_{m},e_{t(m)})$ which is a submodule factorization. As a string, $U = e_{i+1}(\alpha_{i+1} \dots \alpha_j)^k\alpha_{j+1}\dots\alpha_{j}e_j$. Since $k\geq 1$, we have that $\alpha'_{l+1}\dots\alpha'_{m}$ is a substring of $U$. Write $U = (\alpha_{i+1}\dots\alpha'_{l},\alpha'_{l+1}\dots\alpha'_{m},\alpha'_{m+1}\dots\alpha_{j})$. Since $V$ is preinjective, by Remark \ref{rem: prepro and preinj strands defn}, $\alpha'_{l} \in Q_1^{-1}$ and $\alpha'_{m+1} \in Q_1$. Therefore, this admissible pair gives a two-sided graph map, hence by Theorem \ref{thm: basis for ext}, Hom$(U,V)\neq 0 \neq \text{Ext}(V,U)$. Therefore, $U$ and $V$ can't be together in an exceptional set. The proof for when the parameter of $V$ is non-zero follows from this by performing Dehn twists.
\end{proof}

\begin{lem} \label{lem: Preproj and Preinj can't get too long}
Suppose we have an exceptional pair $\{a(i_1,j_1)[\lambda_1], a(i_2,j_2)[\lambda_2]\}$ where both $U = a(i_1,j_1)[\lambda_1]$ and $V=a(i_2,j_2)[\lambda_2]$ are either preprojective or preinjective. Then $\lambda_1 \leq \lambda_2 + 1$ or $\lambda_1 + 1 \leq \lambda_2$. 
\end{lem}

\begin{proof}
We prove the contrapositive. Suppose both modules are preprojective and $\lambda_1 > \lambda_2 + 1$. Write $U = \alpha_{i_1 + 1}\dots \alpha_{i_j}$ and $V = \alpha_{i_2 + 1} \dots \alpha_{j_2} = (e_{s(\alpha_{i_2 + 1})},\alpha_{i_2 + 1} \dots \alpha_{j_2},e_{t(\alpha_{j_2})})$. Then this is a quotient factorization of $V$. Since $\lambda_1 > \lambda_2 + 1$, we have that $\alpha_{i_2 + 1} \dots \alpha_{j_2}$ is a substring of $U$. Write $U = (\alpha_{i_1 + 1}\dots \alpha_{i_2}, \alpha_{i_2 + 1} \dots \alpha_{j_2},\alpha_{j_2+1}\dots\alpha_{j_1})$. Since $V$ is preprojective, by Remark \ref{rem: prepro and preinj strands defn}, $\alpha_{j_2} \in Q_1$ and $\alpha_{j_2+1} \in Q_1^{-1}$. Thus this admissible pair gives a two-sided graph map. We conclude from Theorem \ref{thm: basis for ext} that neither $(U,V)$ nor $(V,U)$ form an exceptional sequence. The proof when both $U$ and $V$ are preinjective follows from duality.
\end{proof}

Putting these Lemmas together, we have shown that if an exceptional collection contains a regular module, this module must be of the form $(i,j;0)$. Moreover, if an exceptional collection contains both preprojective modules and preinjective modules, then they must both be projective and injective. Finally, if an exceptional collection contains only regular modules and preprojective (preinjective) modules, then the winding numbers of the corresponding arcs can differ by at most 1. Before showing that there are only finitely many equivalence classes of exceptional collections under the equivalence relation in Definition \ref{defn: equivalent arc diagrams}, we require one more definition.

\begin{defn}
We call an exceptional collection \textbf{small} if in its corresponding arc diagram, all winding numbers are 0. Two small diagrams are \textbf{equivalent} if they are equivalent under the relation defined in Definition \ref{defn: equivalent arc diagrams}. 
\end{defn}

\begin{thm} \label{thm: bijetion with small diagrams}
There are finitely many equivalence classes of exceptional collections and they are in bijection with equivalence classes of small fundamental diagrams.
\end{thm}

\begin{proof}
We can associate a small diagram to any given equivalence class of exceptional collections as follows. Given an element of an equivalence class $\xi= \{a(i_1,j_1)[\lambda_1], \dots, a(i_n,j_n)[\lambda_n]\}$, either all the $\lambda_k$ are the same, or we have two subsets $A,B \subset \{\lambda_k\}$ such that if $\lambda_l \in A$ then $\lambda_l = \lambda_s \pm 1$ for all $\lambda_s \in B$. Note that by Lemma \ref{lem: Preproj and Preinj can't get too long}, it is either exclusively plus or exclusively minus. If for all $k$, $\lambda_k = z$ for a positive integer $z$, the associated small diagram is the one attained by performing $z$ counter clockwise $2\pi$ Dehn twists. This small diagram contains regular modules and projective modules with winding number $0$. Performing one more counter clockwise $2\pi$ Dehn twist will give another small diagram containing regular modules and injective modules with winding number $0$. Analogously, if  for all $k$, $\lambda_k = -z$, then we also get two small diagrams, one containing regular modules and injective modules and the other containing regular modules and projective modules. In this case, these are the only two small diagrams associated to the equivalence class of $\xi$.

On the other hand, suppose we have two subsets $A,B \subset \{\lambda_k\}$ such that if $\lambda_l \in A$ then $\lambda_l = \lambda_s \pm 1$ for all $\lambda_s \in B$ and let $\lambda = \text{min}(|\lambda_s|,|\lambda_l|)$. Then by performing $\lambda$ $2\pi$ Dehn twists in the appropriate direction, we attain a diagram in which all $\lambda_l\in A$ are one or negative one depending on if the corresponding modules are preprojective or preinjective, and all $\lambda_s \in B$ are zero. If all the $\lambda_l\in A$ equal one, then all the modules in $B$ must also be preprojective by Lemma \ref{lem: Preproj and Preinj can't get too long}. Then one more counter clockwise $2\pi$ Dehn twist gives a small diagram in which the arcs in $B$ are injective with winding number zero, and those in $A$ are projective with winding number $0$. A similar argument holds for the case in which all the $\lambda_l\in A$ are negative one. In this case, this is the unique small diagram in the equivalence class of $\xi$. 

By definition, the association of the equivalence class of $\xi$ to the equivalence class of the aforementioned small diagram is unique; and moreover, each equivalence class of small diagrams is associated to exactly one parametrized family of exceptional collections. Since there are only finitely many small diagrams, the claim holds. 
\end{proof}

\begin{defn}\label{defn: small diagram associated to}
Let $\Xi$ be an equivalence class of exceptional collections. We define \textbf{the small diagram associated to $\Xi$} as follows. 

\begin{itemize}
\item If the small diagram constructed in the proof of Theorem \ref{thm: bijetion with small diagrams} is not unique, the small diagram that contains only regular modules and projective modules is \textbf{the small diagram associated to $\Xi$}. 
\item If the small diagram constructed in the proof of Theorem \ref{thm: bijetion with small diagrams} is unique, then this small diagram is \textbf{the small diagram associated to $\Xi$}. 
\end{itemize}
\end{defn}

With Definition \ref{defn: small diagram associated to} in mind, we see that an equivalence class $\Xi$ under the relation in Definition \ref{defn: equivalent arc diagrams} is parametrized by the number of clockwise $2\pi$ Dehn twists from the small diagram associated to $\Xi$. This motivates the following name.

\begin{defn}\label{defn: parametrized family}
We call an equivalence class of exceptional collections under the relation in Definition \ref{defn: equivalent arc diagrams} a \textbf{parametrized family of exceptional collections}.
\end{defn}

\subsection{An Algebraic Description of Parameterized Families}

\noindent

In this subsection, we will analyze the algebraic consequences of performing twists of the boundary components of the annulus. Let $Q^{\bm\varepsilon}$ be a quiver of type $\tilde{\mathbb{A}}_{n-1}$ with orientation $\bm{\varepsilon}$. Suppose that on the corresponding annulus $A_{Q^{\bm{\varepsilon}}}$, the vertices are written and labeled in clockwise order, respecting the natural numerical order on the vertices. Moreover, that the $p$ vertices on the inner boundary component are equally spaced as well as the $q$ on the outer boundary component. By an \textbf{elementary clockwise twist of the inner boundary component} we mean a clockwise ${2\pi \over p}$ twist of the inner boundary circle. Elementary clockwise twists of the outer boundary component are defined analogously, as are elementary counter clockwise twists of either component. Any twist of this form is called an \textbf{elementary twist} of the annulus.

\begin{lem} \label{lem: algebraic realization of twists}
Let $M = (i,j;l)$ be a string module.
\begin{enumerate}
\item If $M$ is preprojective, we have the following.
\begin{enumerate}
\item An elementary clockwise twist of the inner boundary is equivalent to adding a hook to the end of $M$.
\item If it is possible to delete a hook at the end of $M$, this is equivalent to an elementary counterclockwise twist of the inner boundary.
\item An elementary counter clockwise twist of the outer boundary is equivalent to adding a hook at the start of $M$.
\item If it is possible to delete a hook at the start of $M$, this is equivalent to an elementary clockwise twist of the outer boundary.
\end{enumerate}
\item If $M$ is regular, we have the following.
\begin{enumerate}
\item If $M$ resides in the left tube, an elementary clockwise twist of the outer boundary is equivalent to $\tau M$.
\item If $M$ resides in the left tube, an elementary counterclockwise twist of the outer boundary is equivalent to $\tau^{-1} M$.
\item If $M$ resides in the right tube, an elementary clockwise twist of the inner boundary component is equivalent to $\tau^{-1} M$.
\item If $M$ resides in the right tube, an elementary counter clockwise twist of the inner boundary is equivalent to $\tau M$.
\end{enumerate}
\item If $M$ is preinjective, we have the following.
\begin{enumerate}
\item If it is possible to delete a cohook at the start of $M$, this is equivalent to an elementary clockwise twist of the inner boundary.
\item An elementary counterclockwise twist of the inner boundary is equivalent to adding a cohook at the start of $M$.
\item If it is possible to delete a cohook at the end of $M$, this is equivalent to an elementary counter clockwise twist of the outer boundary.
\item An elementary clockwise twist of the outer boundary is equivalent to adding a cohook at the end of $M$.
\end{enumerate}

\end{enumerate}
\end{lem}

\begin{proof}
This proof relies on the way in which our annulus is set up. Recall that the labeled vertices on the boundary components of the annulus are written in clockwise order respecting the order in which the numbers appear in the quiver. 
\begin{enumerate}
\item Let $M = (i,j;k)$ be a preprojective string and $a(i',j')[k]$ its corresponding arc.
\begin{enumerate}
\item An elementary clockwise twist of the inner boundary component of the annulus keeps $i'$ the same, while sending $j'$ to $s$, where $s$ is the next vertex on the inner boundary component clockwise of $j'$. Thus, we have extended $M$ at its end by adding the next inverse arrow along with all possible subsequent direct arrows, these of course correspond to all the marked points on the outer boundary of the annulus between $j'$ and $s$. Therefore, this operation is equivalent to adding a hook to the end of $M$.
\item If it is possible to delete a hook at the end of $M$, then the string $M$ contains an inverse arrow. Then there is a vertex on the inner boundary component of the annulus $s\neq j'$ over which the arc corresponding to $M$ passes. Let $s$ be the first vertex on the inner boundary component of the annulus counter clockwise of $j'$. An elementary counter clockwise twist of the inner boundary component keeps $i'$ the same, while sending $j'$ to $s$. This is equivalent to removing the last inverse arrow from the string $M$, along with all subsequent direct arrows (the marked points on the outer boundary component). Therefore, this operation is equivalent to deleting a hook at the end of $M$.
\end{enumerate}
The proofs of $(c)$ and $(d)$ are analogous to those of $(a)$ and $(b)$.
\item Suppose $M=(i,j;k)$ is regular and let $a(i',j')[k]$ be its corresponding arc.
\begin{enumerate}
\item Suppose $M$ is left regular, so that $\bm\varepsilon_i = \bm\varepsilon_j = +$. Then an elementary clockwise rotation of the outer boundary sends $j'$ to the next possible vertex $s$ clockwise of $j'$ with $\bm\varepsilon_s = +$. This is equivalent to extending $M$ at the end by a direct arrow while adding all possible subsequent inverse arrows, and hence, equivalent to adding a cohook at the end of $M$. Moreover, this operation sends $i'$ to the next possible vertex $x$ clockwise of $i'$ on the outer boundary component. This is equivalent to removing the first direct arrow from the string $(s,j;k')$, along with all preceding inverse arrows (the marked points on the outer boundary component). Therefore, this operation is equivalent to deleting a hook at the start of $M$. Since $\bm\varepsilon_i = \bm\varepsilon_j = +$, it is impossible to add a cohook at the start of $M$. Therefore by Theorem \ref{thm: tau of string algebra}, this operation is equivalent to $\tau M$.
\end{enumerate}
The proof of $(b)$ is analogous to that of $(a)$ and the proofs of $(c)$ and $(d)$ follow from the fact that the statements are dual to $(a)$ and $(b)$.
\item This case is dual to 1.
\end{enumerate}
\end{proof}

Using this lemma, we can give a combinatorial partial description of $\tau$.

\begin{thm}\label{thm: tau as twists}
Let $\xi = (M_1,M_2,\dots M_n)$ be an exceptional collection of $\Bbbk Q$ modules that does not contain any projectives, and $D_{\xi}$ its corresponding arc diagram. Then the arc diagram $D_{\tau \xi}$ corresponding to the exceptional collection $\tau \xi = (\tau M_1,\tau M_2,\dots \tau M_n)$ is attained from $D_{\xi}$ by an elementary clockwise twist of the outer boundary circle followed by an elementary counter clockwise twist of the inner boundary circle.
\end{thm}

\begin{proof}
Let $M$ be a string module. The case in which $M$ is regular has already been proven in Lemma \ref{lem: algebraic realization of twists}. If $M$ is preprojective and not projective, then the corresponding string contains an inverse arrow. Then the result follows from Lemma \ref{lem: algebraic realization of twists} 1 (b) and (d) along with Theorem \ref{thm: tau of string algebra}. Similarly, if $M$ is preinjective, then this follows from Lemma \ref{lem: algebraic realization of twists} 3 (b) and (d) along with Theorem \ref{thm: tau of string algebra}. 
\end{proof}

Note that, in the case in which it is not possible to delete hooks in projectives, or delete cohooks in injectives, we can still perform elementary twists on the annulus; however, Lemma \ref{lem: algebraic realization of twists} provides no algebraic insight. In these cases, elementary twists send injective modules to projective modules and vice versa, as seen in Example \ref{exam: tranjective component}. To attain some algebraic insight in these cases, we need to recall the equivalence of categories $\nu: ProjQ \rightarrow InjQ$ defined as $DHom(-,\Bbbk Q)$ where $D$ is the duality, also known as the Nakayama functor. This functor sends a projective corresponding to vertex $i$, denoted by $P(i)$, to the corresponding injective at the same vertex, $I(i)$.

\begin{lem}\label{lem: twist is nakayama}
Let $M = (i,j;l)$ be a string module.
\begin{enumerate}
\item If $M$ is preprojective, we have the following.
\begin{enumerate}
\item If it is not possible to delete a hook at the end of $M$, an elementary counter clockwise twist of the inner boundary is equivalent to $\nu M'$ where $M'$ is attained from $M$ by adding a hook to the start of $M$.
\item If it is not possible to delete a hook at the start of $M$, an elementary clockwise twist of the outer boundary is equivalent to $\nu M'$ where $M'$ is attained from $M$ by adding a hook to the end of $M$.
\end{enumerate}
\item If $M$ is preinjective, we have the following.
\begin{enumerate}
\item If it is not possible to delete a cohook at the start of $M$, an elementary clockwise twist of the inner boundary is equivalent to $\nu^{-1} M'$ where $M'$ is attained from $M$ by adding a cohook to the end of $M$.
\item If it is not possible to delete a cohook at the end of $M$, an elementary counter clockwise twist of the outer boundary is equivalent to $\nu^{-1} M'$ where $M'$ is attained from $M$ by adding a cohook to the start of $M$.
\end{enumerate}
\end{enumerate}
\end{lem}

\begin{proof}
We will prove 1. (a). The proof of 1. (b) is analogous and 2. is dual to 1. Let $M = (i,j;l)$ be a preprojective string module such that it is not possible to delete a hook at the end of $M$ and let $a(i',j')[l]$ be its corresponding arc on $A_{Q^{\bm\varepsilon}}$. Thus the string corresponding to $M$ contains no inverse arrows and $l = 0$ since $Q$ is not a cycle. Therefore, $M$ is projective and the string associated to $M$ is $\alpha_{i+1}\alpha_{i+2}\dots \alpha_{j-1}$ in the case $M$ is not simple, and $e_{i+1}$ in the case $M$ is simple. Let $s$ be the marked point on the inner boundary component of $A_{Q^{\bm\varepsilon}}$ that is immediately counter clockwise of $j'$. Then an elementary counter clockwise twist of the inner boundary component sends the arc $a(i',j')[0]$ to the arc $a(s,i')[0]$. The module corresponding to $a(s,i')[0]$ has support only at the marked points $s+1,s+2,\dots,i$, which lie entirely on the outer boundary component of the annulus. Note that it could be the case that $s+1 = i'$. Moreover, since $s+1$ is on the outer boundary component while $s$ is on the inner boundary component, we have that $s+1$ is a source. Therefore the string corresponding to $a(s,i')[0]$ contains only direct arrows $\alpha_{s+1}\dots\alpha_{i-1}$ where $\alpha_{i}$ is a direct arrow. This allows us to conclude that the module $a(s,i')[0]$ is an injective module with simple socle $S(i)$, hence $a(s,i')[0] = I(i)$.

Now consider again the projective $M = (i,j;0)$. Since $M$ contains no inverse arrows, $M = P(i+1)$. In this case, adding a hook at the start of $M$ will give $P(i)$. Let $M'$ denote the module attained by adding a hook to the start of $M$. From this we conclude that $a(s,i')[0] = I(i) = \nu M'$ as desired.
\end{proof}

In the case when $M$ is projective, then $\tau M = 0$; however, an elementary clockwise twist of the outer boundary circle followed by an elementary counter clockwise twist of the inner boundary circle does not annihilate the arc corresponding to $M$. In fact, by Lemmas \ref{lem: algebraic realization of twists} and \ref{lem: twist is nakayama}, this operation sends $M$ to the corresponding injective module at the same vertex. To make this consistent with Theorem \ref{thm: tau as twists}, it seems that it is best to study the bounded derived category of $\Bbbk Q$.

To this end, let $D^b(\Bbbk Q)$ denote the bounded derived category of rep$_\Bbbk Q$. It is known that $D^b(\Bbbk Q)$ is a triangulated category with shift functor $\Sigma^i:D^b(\Bbbk Q) \rightarrow D^b(\Bbbk Q)$. In $D^b(\Bbbk Q)$, almost split triangles that are contained inside a given shift $\Sigma^{i}\text{rep}_\Bbbk Q$ are induced by almost split sequences in rep$_\Bbbk Q$. The others are of the form $\Sigma^{j-1}I \rightarrow E \rightarrow \Sigma^{j}P$ where $P$ is indecomposable projective and $I = \nu(P)$ where $\nu$ is the Nakayama functor; that is, $I$ is the indecomposable injective associated with the same vertex of the quiver. Let $\tau: D^b(\Bbbk Q) \rightarrow D^b(\Bbbk Q)$ denote the equivalence which induces the Auslander-Reiten translation so that $\tau C = A$ if we have a triangle of the form $A\rightarrow B \rightarrow C \rightarrow \Sigma^1 A$. The \textbf{transjective component} $\mathcal{T}$ of the bounded derived category is the full subcategory containing preprojective modules and $\Sigma^{-1}$-shifted preinjective modules [\ref{ref: transjective component}]. We can realize all objects in $\mathcal{T}$ as arcs on the annulus by simply considering the arc associated to the unshifted module. 

In the transjective component, using Lemmas \ref{lem: algebraic realization of twists} and \ref{lem: twist is nakayama}, we can see that an elementary clockwise twist of the outer boundary circle followed by an elementary counter clockwise twist of the inner boundary circle sends a projective $\Bbbk Q$ module $P(i)$ to $\Sigma^{-1}I(i)$, which is $\tau P(i)$ in $\mathcal{T}$. In this setting, the result of Theorem \ref{thm: tau as twists} can be extended to include projectives. Now in the AR quiver of $\mathcal{T}$, we take the convention that for preprojectives, moving down corays corresponds to adding hooks at the end of the string, while moving up rays corresponds to adding hooks at the start of the string. Dually for shifted preinjectives, moving down corays corresponds to deleting cohooks at the start of the string of the unshifted module, where moving up rays corresponds to deleting cohooks at the end of the string of the unshifted module. The shifted injective $\Sigma^{-1} I$ with a downward arrow to the projective $P$ is thus uniquely defined as the shift of the injective module that is gotten by applying the Nakayama functor to the projective that is gotten from $P$ by adding a hook at the start of $P$. Analogously, the shifted injective $\Sigma^{-1} I$ with an upward arrow to the projective $P$ is uniquely defined as the shift of the injective module that is gotten by applying the Nakayama functor to the projective that is gotten from $P$ by adding a hook at the end of $P$. An example of a transjective component with this convention can be seen in Example \ref{exam: tranjective component}.

Finally, we will provide an algebraic description of the parametrized families introduced in Definition \ref{defn: parametrized family}. Let $M$ be an object in $\mathcal{T}$. For $z \in \mathbb{N}$, define $M \pm z$ to be the unique object that lies $z$ positions away from $M$ in the downward (upward) direction on the same coray.

\begin{lem}\label{lem: twists and movements on AR quiver}
Let $M$ be an object in $\mathcal{T}$ and let $a(i,j)[l]$ be the corresponding arc on the annulus. After performing an elementary clockwise twist of the inner circle, the new object is $M + 1$. After performing an elementary counterclockwise twist of the inner circle, the new object is $M-1$.
\end{lem}

\begin{proof}
By Lemma \ref{lem: algebraic realization of twists}, if the object $M$ is preprojective before and after the twist, then the twist corresponds to adding a single hook at the end of $M$, hence a downward arrow in $\mathcal{T}$. Similarly, if the object is preinjective before and after the twist, then the twist corresponds to deleting a cohook at the start of $M$, hence a downward arrow in $\mathcal{T}$. Suppose now that the object begins as an injective $I$ and ends as a projective. Then by Lemma \ref{lem: twist is nakayama}, we have that the projective is attained by adding a cohook to the end of $I$; that is, moving one position from $I$ down the ray on which $I$ lies, then applying $\nu^{-1}$. Thus $P$ is the last term in an almost split sequence in which $I$ is a summand of the middle term with a downward arrow to $P$ in the AR quiver of $\mathcal{T}$. We conclude that $P$ lies one position down a coray from $I$. The proof for counter clockwise twists is similar.
\end{proof}

\begin{thm}
Let $Q$ be a quiver of type $\tilde{\mathbb{A}}_{n-1}$ whose corresponding annulus has $p$ marked points on the inner boundary component. Let $\Xi$ be a parametrized family of exceptional collections. Then $\xi_1,\xi_2 \in \Xi$ if and only if $\xi_1$ and $\xi_2$ contain the same regular modules and there exists a $z \in \mathbb{Z}$ such that for all nonregular $M \in \xi_1$, we have that the unshifted module corresponding to $M + z\cdot p \in \mathcal{T}$ lies in $\xi_2$. 
\end{thm} 

\begin{proof}
This follows from Lemma \ref{lem: twists and movements on AR quiver}.
\end{proof}

\begin{exmp}\label{exam: tranjective component}
Let $Q$ be the quiver $\begin{xymatrix}{1\ar[r] \ar@/^/[rr] & 2 \ar[r] & 3}\end{xymatrix}$. Then we have a parametrized family $\{01,02,21\}$ given by the annulus on the left below. To the right of this annulus is a $2\pi$ clockwise Dehn twist giving another member of the family, namely $\{13,23,21\}$.

\begin{center}
\tikzset{every picture/.style={line width=0.75pt}} %set default line width to 0.75pt        


\end{center}

Depicted below is the transjective component of the Auslander Reiten quiver of $D^b(\Bbbk Q^{\bm{\varepsilon}})$. The aforementioned Dehn twist corresponds to sliding the red oval down the corays to the blue oval, and any other member of this family is attained by sliding the oval along the corays in this fashion.
\begin{center}

\tikzset{every picture/.style={line width=0.75pt}} %set default line width to 0.75pt        

%

\end{center}  

Moreover, using Lemma \ref{lem: algebraic realization of twists}, its proof, and the fact that irreducible morphisms are given by adding hooks or deleting cohooks, we can draw the AR quiver $\Gamma_{\Bbbk Q}$ in terms of arc diagrams as follows. On the left (right) is the preprojective (preinjective) component.
\vspace{-.2cm}
\begin{center}

\tikzset{every picture/.style={line width=0.75pt}} %set default line width to 0.75pt        

%

\end{center}
We can also write the left and right tubes, which are typically drawn between the preprojective and preinjective components, as below where the dotted vertical lines are identified.

\begin{center}

\tikzset{every picture/.style={line width=0.75pt}} %set default line width to 0.75pt        

%

\end{center}

\end{exmp}

\section{Examples/Future Work} \label{sec: examples}

\noindent

For the Kronecker quiver there is only one family and for $\tilde{\mathbb{A}}_2$, the number of parametrized families of exceptional sequences is independent of orientation. This follows from the fact that there is only one possible orientation and it consists of one sink and one source. After writing code to compute the number of families, this does not seem to be the case in general. Further evidence suggesting that the number of families is dependent on orientation comes from the fact that annular non-crossing permutations are dependent on the number of marked points on both the outer and inner circle [\ref{ref: probably depends on orientation}]. We wish to find a formula to count how many families there are in general. For the orientation such that $\varepsilon_i = +$ for all $i \in \{1, 2, \dots , n\}$ and $\varepsilon_0 = -$, we have found that the number of families of exceptional collections for $n = 0,1,2,3,4$ respectively, is given by the sequence 1, 1, 8, 54, 352. When considering $\tilde{\mathbb{A}}_n$ for $n > 0$, dividing each term of this sequence by its respective $n$ gives 1, 4, 18, 88. These are the first few terms of a generalization of the Catalan numbers known as the Rothe numbers [\ref{ref: Master's Thesis Generalized Catalan Numbers}] or the Rothe-Hagen coefficients of the first type [\ref{ref: Gould Generalized Catalan numbers}]. These numbers count the number of lattice paths satisfying certain constraints. In [\ref{ref: Igusa Maresca}], we have proven that these lattice paths are indeed in bijection with families of exceptional collections of type $\tilde{\mathbb{A}}_n$ with the aforementioned orientation. 

The action of changing parameters and Dehn twists in the annulus is equivalent to moving modules along rays or corays in the transjective component of the Auslander-Reiten quiver. For quivers of type $\tilde{\mathbb{A}}_n$, the exceptionality of sequences corresponds to the modules forming certain patterns and shapes in the transjective component. Since all Euclidean quivers are of tame representation type, we wish to determine whether this behavior generalizes and if/how exceptional sequences of modules over any Euclidean quiver algebra can also be classified into finitely many parametrized families. Finally, we would like to see how infinite string modules fit into this picture and possibly extend this construction to string algebras in general.

Throughout the remainder of this section, let $Q^{\bm{\varepsilon}}$ denote the quiver 
\vspace{-.2cm}
\begin{center}
\tikzset{every picture/.style={line width=0.75pt}} %set default line width to 0.75pt        

\begin{tikzpicture}[x=0.75pt,y=0.75pt,yscale=-1,xscale=1]
%uncomment if require: \path (0,332); %set diagram left start at 0, and has height of 332

%Straight Lines [id:da9552607026009816] 
\draw    (313.22,164.01) -- (338.22,164.01) ;
\draw [shift={(340.22,164.01)}, rotate = 180] [color={rgb, 255:red, 0; green, 0; blue, 0 }  ][line width=0.75]    (10.93,-3.29) .. controls (6.95,-1.4) and (3.31,-0.3) .. (0,0) .. controls (3.31,0.3) and (6.95,1.4) .. (10.93,3.29)   ;
%Straight Lines [id:da2642162980971472] 
\draw    (348.22,154.01) -- (335.41,136.62) ;
\draw [shift={(334.22,135.01)}, rotate = 53.62] [color={rgb, 255:red, 0; green, 0; blue, 0 }  ][line width=0.75]    (10.93,-3.29) .. controls (6.95,-1.4) and (3.31,-0.3) .. (0,0) .. controls (3.31,0.3) and (6.95,1.4) .. (10.93,3.29)   ;
%Straight Lines [id:da7604411203785115] 
\draw    (308.22,152.01) -- (320.95,136.55) ;
\draw [shift={(322.22,135.01)}, rotate = 129.47] [color={rgb, 255:red, 0; green, 0; blue, 0 }  ][line width=0.75]    (10.93,-3.29) .. controls (6.95,-1.4) and (3.31,-0.3) .. (0,0) .. controls (3.31,0.3) and (6.95,1.4) .. (10.93,3.29)   ;

% Text Node
\draw (302,154.4) node [anchor=north west][inner sep=0.75pt]    {$1$};
% Text Node
\draw (345,154.4) node [anchor=north west][inner sep=0.75pt]    {$2$};
% Text Node
\draw (324,120.4) node [anchor=north west][inner sep=0.75pt]    {$3$};

\end{tikzpicture}
\end{center}
\vspace{-.2cm}
\noindent
We thus have that $\bm{\varepsilon} = (-, +, +)$. The generators, small strand diagrams, and corresponding fundamental arc diagrams of all 8 parametrized families of exceptional collections (also sequences in this case), are listed in the charts below. The strands that can be lengthened are colored red (preinjectives) or blue (preprojectives), while the black strands are regular and hence can't be lengthened. \\

\begin{center}
\vspace{-.7 cm}
\includegraphics[width=16cm, height = 20cm]{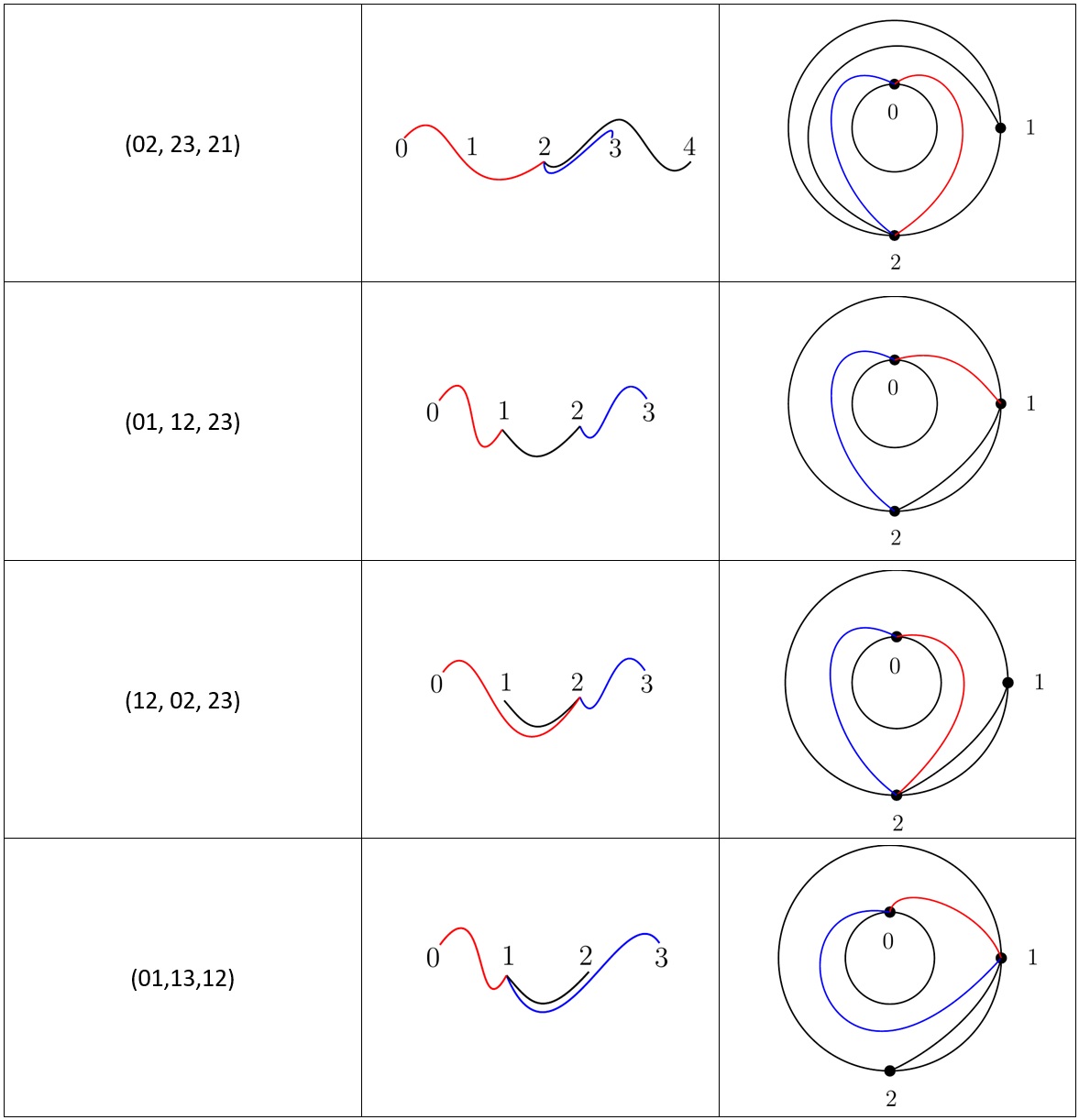}
\includegraphics[width=16cm, height = 20cm]{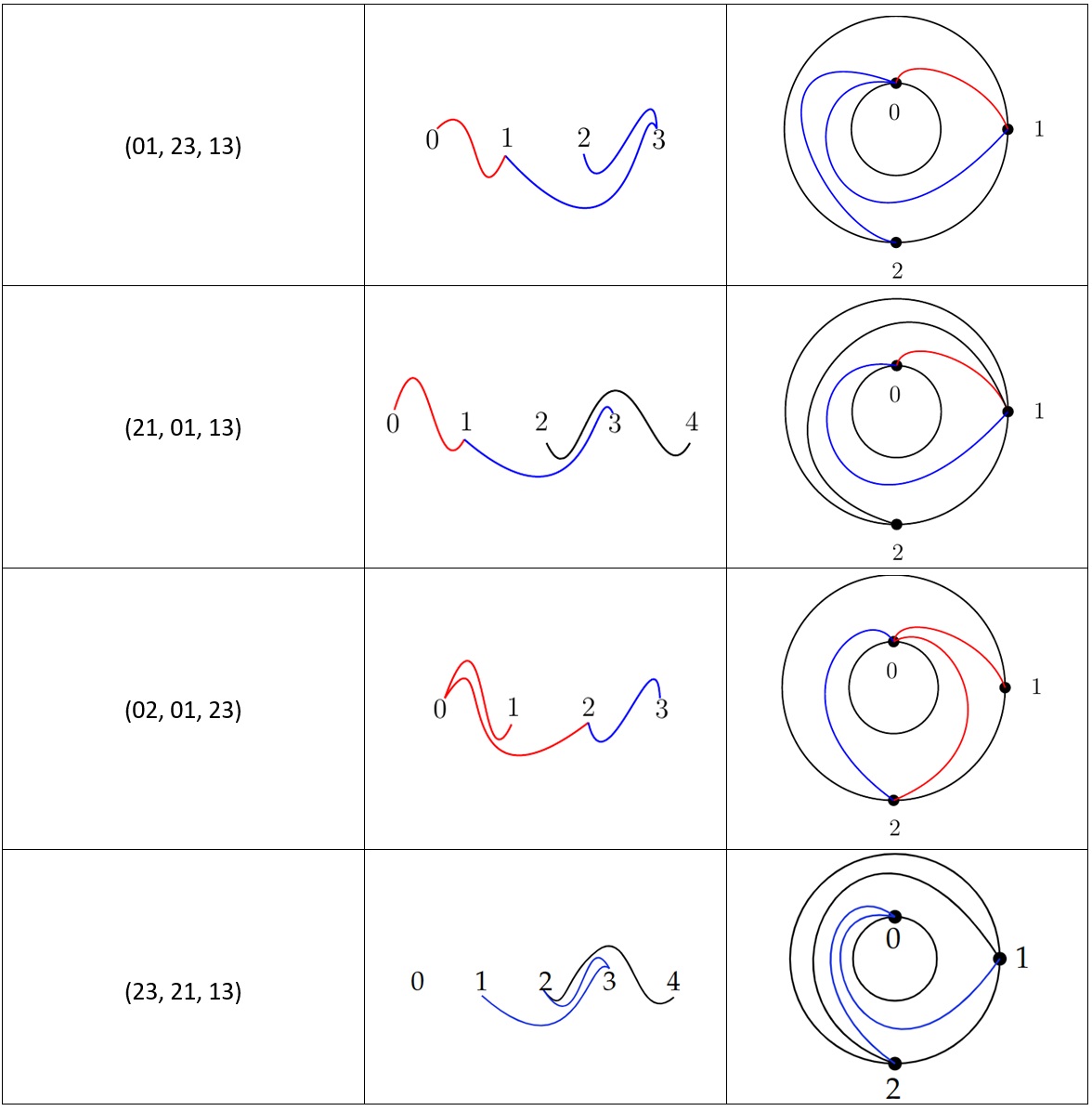}
\end{center}

\pagebreak

\section{References}

\begin{enumerate}[ {[}1{]} ]

\item Araya, T. \textit{Exceptional sequences over path algebras of type $A_n$ and non-crossing spanning trees}. Algebr. Represent. Theory, 16(1):239–250, 2013. \label{ref: chord diagrams 1}

\item Assem, I., Simson, D., Skowronski, A., \textit{Elements of the Representation Theory of Associative Algebras: Volume 1: Techniques of Representation Theory.}. Cambridge University Press. 2006. \label{ref: Blue Book}  

\item Baur, K., Coelho Simões, R. \textit{A Geometric Model for the Module Category of a Gentle Algebra}. International Mathematics Research Notices, Volume 2021, Issue 15, August 2021, Pages 11357–11392, https://doi.org/10.1093/imrn/rnz150 \label{ref: BCS} 

\item Br\"{u}stle, T., Douville, G., Mousavand, K., Thomas, H., Yıldırım, E. \textit{On the Combinatorics of Gentle Algebras}. Canadian Journal of Mathematics (2020), 72(6), 1551-1580. doi:10.4153/S0008414X19000397 \label{ref: BDMTY} 

\item Butler, M. C. R. , Ringel,  C. M., \textit{Auslander-Reiten sequences with few middle terms}. Comm. in Algebra. 15 (1987), 145-179. \label{ref: Butler, Ringel}

\item Crawley-Boevey, W. \textit{Exceptional sequences of representations of quivers}. [MR1206935 (94c:16017)]. In Representations of algebras (Ottawa, ON, 1992),volume 14 of CMS Conf. Proc., pages 117–124. Amer. Math. Soc., Providence, RI, 1993. \label{ref: Bill Exceptional sequences} 

\item Crawley-Boevey, W. \textit{Maps between representations of zero-relation algebras}. J. Algebra 126 (1989), 259-263. \label{ref: CB strings} 

\item Do, N., He, J., Mathews, D. \textit{Counting non-crossing permutations on surfaces of any genus}. The Electronic Journal of Combinatorics. \textbf{4.9} 2021. \label{ref: probably depends on orientation}

\item Fakieh, W. M., Nauman, S. K., Obaid, M. A., Ringel, C. M., and Shammakh, W. S. A. (2013). \textit{The number of complete exceptional sequences for a Dynkin algebra}. Colloquium Mathematicum 133 (2013), 197–210. \label{ref: counting exceptional sequences}

\item Fomin, S., Shapiro, M., Thurston, D. \textit{Cluster Algebras and Triangulated Surfaces Part I: Cluster Complexes}. Preprint, arXiv:0608367v3 [math.RA], 2007. \label{ref: clusters are in bijection with triangulations} 

\item Gabriel, P. \textit{The universal cover of a representation-finite algebra}. Lect. Not. Math. 903(1981), 68-105. \label{ref: Gabriel Universal Cover} 

\item Garver, A., Igusa, K., Matherne, P., J., Ostroff, J. \textit{Combinatorics of Exceptional Sequences in Type A}. The Electronic Journal of Combinatorics. \textbf{26} 2019. \label{ref: combinatorics exceptional sequences type A} 

\item Garver, A., Matherne, P., J. \textit{A combinatorial model for exceptional sequences in type A}. 27th International Conference on Formal Power Series and Algebraic Combinatorics (FPSAC 2015), Jul 2015, Daejeon, South Korea. pp.393-404. ffhal-01337808f. \label{ref: chord diagrams 2}  

\item Gorodentsev, A. L. and Rudakov, A. N. \textit{Exceptional vector bundles on projective spaces}. Duke. math. J. 54 (1987), 115–130. \label{ref: vector bundles over projective plane 10} 

\item Gorsky, E., Gorsky, M. \textit{A braid group action on parking functions}. arXiv preprint arXiv:1112.0381. \label{ref: bijection with parking functions}

\item Gould, H. W. \textit{Fundamentals of Series, eight tables based on seven unpublished
manuscript notebooks}. (1945–1990), edited and compiled by J. Quaintance (2010) \label{ref: Gould Generalized Catalan numbers}

\item Igusa, K., Maresca R. \textit{On Clusters and Exceptional sets in Types $\mathbb{A}$ and $\tilde{\mathbb{A}}$}. Preprint, arXiv:2209.12326 [math.RT], 2022 \label{ref: Igusa Maresca} 

\item Igusa, K., Sen, E. \textit{Exceptional Sequences and Rooted Labeled Forests}. Preprint, arXiv:2108.11351 [math.RT], 2021. \label{ref: Igusa Sen rooted forests} 

\item Looijenga, E. \textit{The complement of the bifurcation variety of a simple singularity}. Inventiones mathematicae 23.2 (1974): 105–116. 

\item Laking, R. \textit{String Algebras in Representation Theory}. [Thesis]. Manchester, UK: The University of Manchester; 2016. \label{ref: String Algebra Info} 

\item Nguyen, C. Van, Todorov, G. Zhu, S. \textit{Preprojective algebras of tree-type quivers}. Preprint, arXiv:1612.01585 [math.RA] \label{ref: transjective component}

\item Richardson, S. L. Jr., \textit{Enumeration of the generalized Catalan numbers}. M.Sc. Thesis, Eberly College of Arts and Sciences, West Virginia University, Morgantown, West Virginia (2005) \label{ref: Master's Thesis Generalized Catalan Numbers}

\item Ringel, M., C. \textit{The braid group action on the set of exceptional sequences of a hereditary Artin algebra}. Abelian group theory and related topics (Oberwolfach, 1993), volume 171 of Contemp. Math., pages 339–352. Amer. Math. Soc., Providence, RI, 1994. \label{ref: braid group acts on excep seq}

\item Rudakov, A.N. \textit{Exceptional collections, mutations and helices}. Helices and Vector Bundles, London Math. Soc. Lecture Note Ser., vol. 148, pp. 1–6. Cambridge Univ. Press, Cambridge (1990) \label{ref: vector bundles over projective plane}

\item Schiffler, R. \textit{Quiver Representations}. CMS Books in Mathematics, Springer International Publishing (2014) \label{ref: Schiffler Quiver Reps}

\item Schr\"{o}er, J. \textit{Modules without self-extensions over gentle algebras}. J. Algebra 216 (1999), no. 1, 178-189. \label{ref: Schroer}

\item  Seidel, U. \textit{Exceptional sequences for quivers of Dynkin type}. Communications in Algebra, 29(3),
(2001), 1373–1386 \label{ref: Exceptional sequences dynkin type} 

\item Simson, D., Skowroński, A. \textit{Elements of the Representation Theory of Associative Algebras Volume 2: Tubes and Concealed Algebras of Euclidean type}. Cambridge University Press. 2007. \label{ref: blue book 2}

\item Warkentin, M. \textit{Fadenmoduln $\ddot{u}$ber $\tilde{A}_n$ und Cluster-Kombinatorik}. Diploma Thesis, University of Bonn. Available from http://nbn-resolving.de/urn:nbn:de:bsz:ch1-qucosa-94793 (2008) \label{ref: Master's Thesis Clusters and Triangulations} 

\end{enumerate}

\end{document}